\documentclass[12pt]{article}

\usepackage{amsmath,amssymb,amsthm,algorithm,algpseudocode,nicematrix,hyperref}

\newcommand{\cur}{\mathrm{curl}\,}
\newcommand{\id}{\mathrm{d}}
\newcommand{\dx}{\ \mathrm{d}x}

\newcommand{\ds}{\ \mathrm{d}s}
\newcommand{\J}{\mathcal{J}}
\newcommand{\mH}{\mathcal{H}}
\newcommand{\mL}{\mathcal{L}}
\newcommand{\E}{\mathcal{E}}
\newcommand{\T}{\mathcal{T}}
\newcommand{\R}{\mathbb{R}}
\newcommand{\I}{\mathcal{I}}
\newcommand{\V}{\mathcal{V}}
\newcommand{\C}{\mathcal{C}}
\newcommand{\M}{\mathcal{M}}

\newcommand{\boldu}{\boldsymbol{u}}
\newcommand{\boldp}{\boldsymbol{p}}
\newcommand{\boldv}{\boldsymbol{v}}
\newcommand{\boldl}{\boldsymbol{\lambda}}
\newcommand{\boldm}{\boldsymbol{\mu}}
\newcommand{\bolde}{\boldsymbol{\eta}}

\theoremstyle{plain}

\newtheorem{lemma}{Lemma}

\theoremstyle{definition}
\newtheorem{definition}{Definition}

\theoremstyle{remark}
\newtheorem{remark}{Remark}
\newtheorem{example}{Example}

\begin{document}

\setcounter{page}{1}

\title{Multi-material topology optimization of electric machines under maximum temperature and stress constraints}
\author{Peter~Gangl$^1$         \and
        Nepomuk~Krenn$^1$ \and
Herbert~De~Gersem$^2$
}
\date{$^1$RICAM, Austrian Academy of Sciences, \\
Altenberger Stra{\ss}e 69, 4040 Linz, Austria\\
$^2$Institute for Accelerator Science and Electromagnetic Fields, TU Darmstadt, \\
Schloßgartenstr. 8, 64289 Darmstadt, Germany}

\maketitle
\begin{abstract}
The use of topology optimization methods for the design of electric machines has become increasingly popular over the past years. Due to a desired increase in power density and a recent trend to high speed machines, thermal aspects play a more and more important role. In this work, we perform multi-material topology optimization of an electric machine, where the cost function depends on both electromagnetic fields and the temperature distribution generated by electromagnetic losses. We provide the topological derivative for this coupled multi-physics problem consisting of the magnetoquasistatic approximation to Maxwell's equations and the stationary heat equation. We use it within a multi-material level set algorithm in order to maximize the machine's average torque for a fixed volume of permanent-magnet material, while keeping the temperature below a prescribed value. Finally, in order to ensure mechanical stability, we additionally enforce a bound on mechanical stresses. Numerical results for the optimization of a permanent magnet synchronous machine are presented, showing a significantly improved performance compared to the reference design while meeting temperature and stress constraints.

\end{abstract}
\section{Introduction} 
Electric machines play an important role in energy transition. Since the geometry of an electric machine has a strong impact on its performance, design optimization methods are often used in practice when the machine is designed for a particular purpose. A widely used approach is to represent the geometry by parameters like lengths, radii or angles and to optimize these parameters, either by derivative free \cite{Kim2024} or gradient-based \cite{Wiesheu2024} optimization methods. In such a setting, the space of possible designs remains limited by the chosen parametrization. In order to overcome these design limitations, over the past decade, topology optimization methods \cite{BendsoeSigmund2003a, AllaireJouveDapogny2021} have been applied to more and more optimal design problems from electromagnetics \cite{LucchiniTorchioCirimeleAlottoBettini2022} and, in particular, to the design optimization of electric machines \cite{Cherriere_2022aa, Kuci2021, AmstutzGangl2019}. 

A common quantity of interest for an electric machine is the average torque produced for a given input current. Thus, often the goal is to find an optimized material arrangement in the rotor (often consisting of ferromagnetic material, air and permanent magnets) or stator (ferromagnetic material, air, copper) that maximizes the average torque while keeping the amount of permanent magnet material low for cost reasons. Since electromagnetic design optimization may often result in geometries that are not mechanically stable at the designated rotation speed, several authors have combined it with mechanical constraints such as local stress constraints \cite{Lee_MTPAMech, Li_2023aa, Holley_Diss} or a compliance constraint \cite{GuoBrown2020}.

A current goal in electric machine design is to increase power density, i.e., to reduce volume and mass while maintaining the same electromagnetic performance. Because increased heat loss densities should be compensated by increased cooling heat flux density, thermal aspects in electric machines are becoming increasingly important. The main heat source in an electric machine are the copper losses coming from the currents in the stator coils, which can be effectively cooled from outside. The eddy-current losses in the rotor's permanent magnets, where cooling is more challenging, need a more careful consideration. Due to the quadratic scaling with the rotational speed, these losses get even more relevant in high-speed applications. There are various approaches on how to model permanent magnet eddy-current losses, see \cite{OumaraDubas_Eddy_Review}. In \cite{Yamazaki_Teeth} and \cite{Zhu_LossReduction}, the authors optimized the shape of the stator teeth, in \cite{Yamazaki_Magnets}, they considered the basic magnet topology and the shape of the air pockets inside the rotor to maximize machine efficiency considering permanent magnet eddy-current losses. With this method, it is not possible to control the temperature inside the machine directly, which is desirable, since the permanent magnets are very sensitive to heat. To compute the temperature distribution explicitly, one has to solve an additional heat equation. Here, it is crucial to choose the correct boundary conditions to incorporate cooling effects along the shaft and the air gap. Compared to the electromagnetic simulation, which includes solving a nonlinear equation for multiple rotor positions, the computational overhead for this additional equation is moderate. 

The common design optimization workflow in industry is a sequential one, meaning to first optimize electromagnetic performance and later dealing with temperature bounds, possibly sacrificing some of the gained performance. In contrast, an integrated optimization approach could account for the interaction between these fields already during the optimization. In the past years, multiphysics topology optimization has been studied by many researchers, see \cite{KambampatiGrayKim2020} for battery pack optimization under stress and temperature constraints; \cite{AmigoGiustiNovotny2016} for piezoelectric actuators; \cite{AllaireGfrerer2024} where, among other applications, fluid structure interaction is considered; or \cite{Neofytou2025} for acoustic-structure interaction. Nevertheless, so far, only few works dealing with integrated optimization accounting for the interaction of electromagnetic and thermal fields in electric machines have been reported in the literature, see e.g. \cite{Babcock_2023aa, Babcock2024} for the optimization of geometric parameters. Topology optimization in this setting seems to be entirely unexplored.

The most widely used classes of topology optimization methods are density-based methods \cite{BendsoeSigmund2003a} employing a spatially continuous design variable that smoothly interpolates between the materials of interest on the one hand, and level-set based methods \cite{AllaireJouveDapogny2021} on the other hand. In the latter class of methods, the design is represented by the values of a continuous level set function where no material interpolation is needed if material interfaces are resolved by the discretization. Level set based methods can be driven by either shape derivatives \cite{OsherSethian_1988a, AllaireJouveToader2004} or topological sensitivity information \cite{Amstutz_Levelset, YamadaIzuiNishiwakiTakezawa2010}. All of these classes of methods have been extended to the case of multiple materials, see \cite{Sigmund2001,Cherriere_2022aa} for density-based methods, or \cite{AllaireDapognyDelgadoMichailidis2014, Laurain2024} and \cite{NodaNoguchiYamada2022, Gangl_Multi} for level set methods based on shape or topological sensitivities, respectively. In this work, we will employ the multi-material level set method based on topological derivatives \cite{Gangl_Multi} which by construction does not feature large regions of intermediate materials and can deal with topology changes in a flexible way.

 Topological derivatives represent the sensitivity of a shape function with respect to local material perturbations. There exist different methods for obtaining topological derivatives of optimization problems constrained by linear or semilinear partial differential equations (PDEs), e.g. the topological-shape sensitivity method \cite{Novotny2013}, or the approaches presented in \cite{Amstutz_Sensitivity} or \cite{Sturm_2020}. While in this setting, fundamental solution techniques can be used to obtain closed-form formulas of topological derivatives, this is no longer possible in the case of quasilinear PDE constraints. Then, however, the topological derivative can be obtained in dependence of a corrector term appearing in the asymptotics of the perturbed PDE, which can be numerically computed in an offline phase; see \cite{AmstutzGangl2019, Gangl_Simplified, Gangl_Automated}. 
 
 In practice, one is often interested in optimized designs which satisfy local state constraints such as local temperature or mechanical stress constraints. Such constraints can be treated by quadratic penalty functionals as recently presented in \cite{AndradeNovotnyLaurain2024}, or penalty functionals based on $p$-norms as presented in \cite{Amstutz_Constraint, Amstutz_vonMises} in the case of constraints on derivatives of the state variable. 

 The rest of the manuscript is organized as follows: Section \ref{sec_modelproblem} introduces the model problem considered throughout this work. The concept of topological derivatives and the considered multi-material level set algorithm are introduced in Section \ref{sec_TDmultmat}. We then present how to compute the topological derivative of a purely electromagnetic problem in Section \ref{sec:T} before going over to the electromagnetic-thermal coupled problem in Section \ref{sec:TCt}. In Section \ref{sec_mechanics} we add the aspect of mechanical stresses before presenting numerical results in Section \ref{sec_numerics}.

% \textcolor{blue}{Irgendwo einbauen: Multiphysics is becoming more and more interesting; (\href{https://doi.org/10.1016/j.compstruc.2020.106265}{A. Kim battery (mech+thermal) KambampatiGrayKim2020}),  (\href{https://doi.org/10.1137/151004860}{Giusti piezoelectric} AmigoGiustiNovotny2016), (\href{https://link.springer.com/article/10.1007/s00158-024-03917-5}{Allaire fluid-structure interaction} AllaireGfrerer2024), acoustic-structure A. Kim \cite{Neofytou2025} }

% To be cited
% \begin{itemize}
%     \item \href{https://www.scipedia.com/wd/images/9/95/Draft_Sanchez_Pinedo_266896281pap_1462.pdf}{TOPOLOGICAL DERIVATIVE-BASED HEAT SINK DESIGN WITH TEMPERATURE CONSTRAINTS} (Andrade, Laurain, Novotny)
%     \item \href{https://www.researchgate.net/publication/371431723}{Multi-Disciplinary Design Optimization of an Electric Motor Considering Thermal Constraints}(applied, aircraft) available \href{https://www.researchgate.net/profile/Jason-Hicken/publication/371431723_Multi-Disciplinary_Design_Optimization_of_an_Electric_Motor_Considering_Thermal_Constraints/links/6485c55b2cad460a1b0c37d4/Multi-Disciplinary-Design-Optimization-of-an-Electric-Motor-Considering-Thermal-Constraints.pdf?_sg%5B0%5D=7Kfgpt_xOLLrB12038D8E42xtgosERFcudCJSUa2ESfnLnCfIKyZ7XwvdQbHOILinIUfGnqB8TtmABT2Xi0ipg.mugHsG1BwxfPmBYqboafsV4KgJKaCHUU3BLNtUdbCtjWsfSeVgpYNRF4V0C_pqUBGJUuciHYoZTJgRDXk0NKJA&_sg%5B1%5D=MG3Smo9gEtfXT48dbt5_-kORcL1Im5jdtIT4sJgPY7-o83ONvuyb7kTd-YwjK8T65CllhmdRBXsRvde92AZF5fEAxi16eoOo3h99Bs8WivcY.mugHsG1BwxfPmBYqboafsV4KgJKaCHUU3BLNtUdbCtjWsfSeVgpYNRF4V0C_pqUBGJUuciHYoZTJgRDXk0NKJA&_sg%5B2%5D=f240iCwn7fJaVX27GHmShFRPCXgn3r0NzgXQ_nkXMDla69xpkWkPqENhFpk_op0ivBuwSG2ZMkD4gqc.pSwHnmrFxOmMNR1gY6WavS-bRUU1WwzEUOxq1rCeE2aBC7388llkElnLCbpbO4XPA5YVPhAM8wUCmp36PXvTwg&origin=publicationDetail_similarResearch&_rtd=eyJjb250ZW50SW50ZW50Ijoic2ltaWxhciJ9&_tp=eyJjb250ZXh0Ijp7ImZpcnN0UGFnZSI6InB1YmxpY2F0aW9uIiwicGFnZSI6InB1YmxpY2F0aW9uIiwicG9zaXRpb24iOiJwYWdlQ29udGVudCJ9fQ}{here on researchgate} \\
%     Follow-up: \href{https://arc.aiaa.org/doi/epdf/10.2514/1.J063627}{Electrothermal Coupling Methodologies and Their Influence on the Optimization of Aircraft Electric Motors}
%     \item \href{https://ieeexplore.ieee.org/document/10700505}{Topology Optimization for Enhancing Electric Machine Performance: A Review}
%     \item \href{https://doi.org/10.1007/s00542-024-05699-8(0123456789().,-volV)(0123456789(). ,- volV)}{Optimal shape design to improve torque characteristics of interior
% permanent magnet synchronous motor for small electric vehicles}, Hyeon-Jun Kim, Soo-Whang Baek, 2024
%     \item Multiphysics topology optimization (mechanics and thermal for batteries):
%     \href{https://www.sciencedirect.com/science/article/pii/S0045794920300687}{Level set topology optimization of structures under stress and temperature constraints}
    
% \end{itemize}

\subsection*{Notation}
Let us briefly collect some notation used in this work. For open domains $D\subset\R^2$ we denote by $L_2(D), H^1(D)$ the spaces of square integrable functions and functions with square integrable first weak derivatives. By $B_\epsilon(z):=\{x\in\R^2:\|x-z\|<\epsilon\}$ we denote the open ball with center $z$ and radius $\epsilon$. The function $\chi_A\in L_2(D)$ is the characteristic function of a set $A\subset\R^2$. We denote the two-dimensional unit vector in direction $\varphi$ by $e_\varphi:=(\cos\varphi,\sin\varphi)^T$ and the rotation matrix $R_\varphi:=(e_\varphi,e_{\varphi+\frac{\pi}{2}})$. Further, we abbreviate with $\rho_\varphi:\R^2\rightarrow\R^2,(x,y)^T\mapsto R_\varphi(x,y)^T$ a rotational coordinate transformation. There are two variants of curl operators in 2d. By $\cur(v)=(\partial_yv,-\partial_xv)^T=R_{-\frac{\pi}{2}}\nabla v$ we denote the scalar-to-vector curl, by $\widetilde{\cur}(v)=-\partial_yv_x+\partial_xv_y=\mathrm{div}(R_{-\frac{\pi}{2}}v)$ the vector-to-scalar counterpart.
\section{Model problem} \label{sec_modelproblem}
\begin{figure}
    \centering
    \includegraphics[width=0.4\textwidth]{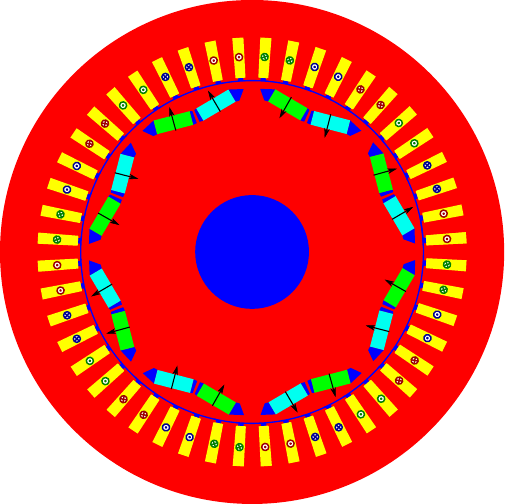}
    \includegraphics[width=0.5\textwidth]{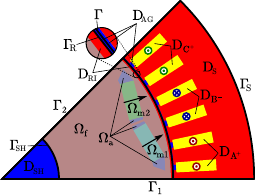}
    \caption{2d cross section of the reference machine (left), one pole of the reference machine $D_\mathrm{all}$ with radial boundaries $\Gamma_1, \Gamma_2$ consisting of rotor, stator and air gap (right). Stator consists of iron $D_S$ (red) with outer boundary $\Gamma_S$ and copper coils $D_{A^+},D_{B^-},D_{C^+}$ (yellow). The air gap $D_\mathrm{AG}$ (blue) is split by circle $\Gamma$ used for harmonic mortaring. The rotor consists of shaft $D_\mathrm{SH}$ (blue) with boundary $\Gamma_\mathrm{SH}$, design domain $D$ (greyed out) and iron ring $D_\mathrm{RI}$ (red) with outer boundary $\Gamma_R$. Design domain consists of iron $\Omega_f$ (red), air $\Omega_a$ (blue) and permanent magnets $\Omega_{m_1},\Omega_{m_2}$ (light blue, light green).}
    \label{fig:setup}
\end{figure}
\begin{table}
    \centering
    \begin{tabular}{l|l}
         Number of pole pairs $N_\mathrm{pp}$ and slots&  4, 48\\
         Inner and outer radius rotor& 26.5\,mm, 78.6\,mm  \\
         Inner and outer radius stator& 79.1\,mm, 116\,mm  \\
         Width of iron ring and air gap&1\,mm, 0.5\,mm\\
         Axial length&179\,mm\\
         Amplitude of current density $\hat{j}$, load angle $\phi_0$&$23.7\,\mathrm{A/mm^2}$, $6^\circ$\\
         % Number of turns per slot, fill factor & 60, 0.6\\
         Iron bh-curve parameters $\nu_f,K_f,N_f$&200\,m/H, 2.2\,T, 12\\
         Vacuum and permanent magnet reluctivity $\nu_0,\nu_m$&$10^7/(4\pi)\,\mathrm{m/H},\nu_0/1.086$\\
         Remanent flux density $B_R$& 1.216\,T \\
         Magnetization direction $\varphi_1,\varphi_2$&  $30^\circ, 15^\circ$\\
         Rotational speed $V_\mathrm{rot}$ & 16000\,$\mathrm{rpm}$\\
         Electric conductivity of magnet $\sigma_m$&$6.7\times 10^5$\,S/m\\
         Thermal conductivities $\lambda_f,\lambda_m,\lambda_a$&16, 9, 0.05\,W/(mK)\\
         Thermal heat transfer coefficients $h_{SH},h_{AG}$&0.235, 260\,$\mathrm{W/(m^2K)}$\\
         Ambient temperature $\vartheta_0$&$40^\circ$C\\
         Mass densities $\rho_f,\rho_m$&7.65, 8.4\,$\mathrm{g/cm^3}$\\
         Young's modulus of iron $E_f$&200\,GPa
    \end{tabular}
    \caption{Machine geometry and material data}
    \label{tab:data}
\end{table}
Electromagnetic phenomena are described by Maxwell's equations, see e.g. \cite{Griffiths_Electrodynamic},
\begin{align}\label{eq:Maxwell}
    \cur H=J+\partial_tD,\quad\cur E=-\partial_tB,\quad \mathrm{div}B=0,\quad\mathrm{div}D=\rho,
\end{align}
where the magnetic field $H$, the magnetic flux density $B$, the electric field $E$, the electric flux density $D$, the current density $J$ are vector-valued functions and the electric charge density $\rho$ a scalar-valued function in space and time. In low frequency applications, as electric machines, the displacement currents $\partial_t D$ can be neglected. We consider the constitutive relations
\begin{align}
    H=h(B),\quad J=J_s+\sigma E,
\end{align}
where $h$ is the magnetic material law, $J_s$ the impressed current density and $\sigma$ the electric conductivity. Since $B$ is divergence-free, there exists a vector-valued function $A$ called the magnetic vector potential with $B=\cur A$. Substitution now yields the magnetoquasistatic equation
\begin{align}
    \sigma\partial_t A+\cur h(\cur A)=J_s,
\end{align}
also called the eddy-current problem. As it is common practice in electric machine optimization, we assume our magnetic quantities to be constant in axial direction and consider only a 2d cross section as computational domain. This yields a scalar-valued problem
\begin{align}
    \sigma\partial_t u+\widetilde{\cur}h(\cur u)&=j\\
\end{align}
where $u(t,x,y)$ and $j(t,x,y)$ are the third components of the magnetic vector potential $A=(0,0,u)^T$ and the source current density $J_s=(0,0,j)^T$, respectively.
\begin{remark}
    In our two-dimensional setting, this equation has for an isotropic material $h(b)=f(|b|)b$ the same structure as a quasilinear heat equation
    \begin{align}
        \sigma\partial_tu-\mathrm{div}(f(|\nabla u|)\nabla u)=j
    \end{align}
    since $\widetilde{\cur},\cur$ are rotated versions of $\mathrm{div},\nabla$, respectively.
\end{remark}
The machine under investigation, a permanent magnet synchronous machine, is displayed in Figure \ref{fig:setup}. The outer part of the machine consists of an iron stator $D_S$ with copper coils $D_{A^+}, D_{B^-}, D_{C^+}$, where the three-phase excitation current density $j$ is applied. In the inner we have the rotor $D_R$ which rotates with the same speed as the magnetic field induced by the excitation current. It consists of a shaft $D_\mathrm{SH}$ in the middle, to which the mechanical forces are transferred, an outer iron ring $D_\mathrm{RI}$, iron $\Omega_f$, air pockets $\Omega_a$ and permanent magnets $\Omega_{m_1},\Omega_{m_2}$ with magnetization direction $\varphi_1,\varphi_2$, respectively. The latter domains will change throughout the optimization. We denote the actual material configuration by the tuple
 \begin{align}\label{eq:configuration}
    \Omega=(\Omega_f,\Omega_a,\Omega_{m_1},\Omega_{m_2}),
\end{align}
where $\Omega_f,\Omega_a,\Omega_{m_1},\Omega_{m_2}$ are open subsets of the design domain $D=D_R\setminus\overline{D_\mathrm{SH}\cup D_\mathrm{RI}}$ and form a disjoint partition of $D$. The iron ring $D_\mathrm{RI}$ is not part of the design domain due to mechanical reasons and to maintain a smooth rotor surface at the air gap. Rotor $D_R$ and stator $D_S$ are magnetically connected via a thin air gap $D_\mathrm{AG}$. The iron parts in rotor and stator are laminated, i.e., they are formed by a stack of thin isolated iron sheets. The dimensions of the machine and other characteristic numbers are presented in Table \ref{tab:data}. We are interested in steady state operation of the machine, i.e. all quantities are periodic with a period $T_\mathrm{full}>0$. A common assumption is that no magnetic flux leaves the machine, which corresponds to a homogeneous Dirichlet boundary condition for the magnetic vector potential. Due to symmetries, it is sufficient to simulate only one pole of the machine, which is, in our case, a piece of $45^\circ$, called $D_\mathrm{all}$, with anti-periodic boundary conditions on $\Gamma_1,\Gamma_2.$ Using these simplifications, the model problem reads: Find $u$ with
\begin{align}\label{eq:magnetoquasistatics}\begin{aligned}
    \sigma_{\Omega(t)}\partial_tu(t)+\widetilde{\cur}_xh_{\Omega(t)}(\cur_x u(t))&=j(t)&&\text{ in }D_\mathrm{all}\\
    u(t)&=0&&\text{ on }\Gamma_S\\
    u(t)|_{\Gamma_1}&=-u(t)|_{\Gamma_2},&&\\
    \sigma_{\Omega(0)} u(0)&=\sigma_{\Omega(T)} u(T)&&\text{ on }D_\mathrm{all}
\end{aligned}
\end{align}
for $t\in [0,T]$. The rotation of the rotor results in a time-dependent rotor geometry $\Omega(t)$ which is incorporated via the the dependence of the material laws on the actual material configuration
\begin{align}\label{eq:material_laws}\begin{aligned}
h_\Omega(b)&=h_f(b)\chi_{D_S\cup D_\mathrm{RI}\cup\Omega_f}+h_a(b)\chi_{D_{A^+}\cup D_{B^-}\cup D_{C^+}\cup D_\mathrm{AG}\cup D_\mathrm{SH}\cup\Omega_a}+h_{m_1}(b)\chi_{\Omega_{m_1}}+h_{m_2}(b)\chi_{\Omega_{m2}},\\
    h_f(b)&=\nu_0 b+(\nu_f-\nu_0)\frac{K_f}{\sqrt[N_f]{K_f^{N_f}+|b|^{N_f}}}b,\quad
    h_a(b)=\nu_0 b,\quad
    h_{m_i}(b)=\nu_m(b-B_Re_{\varphi_i}), i=1,2,\\
    \sigma_\Omega&=\sigma_m\chi_{\Omega_{m_1}\cup\Omega_{m_2}},
\end{aligned}
\end{align}
with reluctivities $\nu_0,\nu_m$, iron bh-curve parameters $\nu_f,K_f,N_f$, remanent flux density $B_R$, magnetization directions $\varphi_1,\varphi_2$ and electric conductivity $\sigma_m$ as in Table \ref{tab:data}. The conductivity in the iron parts is $0$ due to lamination. The excitation currents $j(t)$ are periodic by $T_\mathrm{full}=\tfrac{1}{N_\mathrm{pp}V_\mathrm{rot}}$, which is the periodicity of the fields in the stator. The fields in the rotor are periodic by 
\begin{align}\label{eq:periodicity}
    T=\frac{T_\mathrm{full}}{6},
\end{align}
since the rotor device is moving synchronously with the fields induced by $j$ and only sees the variations due to the geometric inhomogeneities of the stator.
Since $\sigma_\Omega$ is only non-zero in the permanent magnets, which are part of the rotor, $u$ inherits this higher periodicity which corresponds for our machine to a mechanical rotation by $15^\circ.$

\section{Multi-material level set based topology optimization by the topological derivative} \label{sec_TDmultmat}
We recall the framework from \cite{Gangl_Multi} for the multi-material design optimization of some quantity $\J(\Omega)$. We aim to distribute $M$ different materials denoted by
\begin{align}
    \Omega=(\Omega_i)_{i\in\M},\quad \M:=\{1,...,M\}
\end{align}
and denote by $z\in\Omega:\Leftrightarrow z\in\cup_{i\in\M}\Omega_i$ that a point $z\in D$ is not on the material interfaces.
\subsection{Topological perturbations}
In order to introduce the topological derivative as sensitivity information for our optimization, we need the notion of topological perturbations in the multi-material setting.
\begin{definition}
    Let $z\in\Omega_i$ for some $i\in\M$. We call $\omega_\epsilon(z):=z+\epsilon\omega$ a perturbation in point $z$ of size $\epsilon>0$ and shape $\omega\subset\R^2$, and
    \begin{align}\label{eq:TDpert}
        \Omega_\epsilon^{i\rightarrow j}(z):=(\tilde{\Omega}_k)_{k\in\M},\quad \tilde{\Omega}_k=\begin{cases}
            \Omega_i\setminus\overline{\omega_\epsilon(z)}&\text{ if }k=i\\
            \Omega_j\cup\omega_\epsilon(z)&\text{ if }k=j\\
            \Omega_k&\text{ else}
        \end{cases}
    \end{align}
    for $k,j\in\M\setminus\{i\}$ the $i$-to-$j$-topologically perturbed configuration. 
\end{definition}

In words, we fill the perturbation $\omega_\epsilon(z)$ with a different material $j\neq i$, thus removing it from $\Omega^i$ and adding it to $\Omega^j$.
This leads us to the definition of the multi-material topological derivative. First, we define the sensitivity for a fixed change from material $i$ to material $j$.
\begin{definition}
    Let $i,j\in \M, j\neq i, z\in\Omega_i$ and $\Omega \mapsto \J(\Omega)$ a real-valued shape function. The topological derivative of $\J$ in $z$ for changes from material $i$ to material $j$ is defined by
    \begin{align}\label{eq:TDdef}
        \id^{i\rightarrow j}\J(\Omega)(z)=\lim_{\epsilon\rightarrow 0}\frac{\J(\Omega_\epsilon^{i\rightarrow j}(z))-\J(\Omega)}{|\omega_\epsilon(z)|},
    \end{align}
    where $\Omega_\epsilon^{i\rightarrow j}(z)$ is the $i$-to-$j$-topologically perturbed material configuration. Since there are $M-1$ possible choices of $j$, we collect them in the vector-valued topological derivative for $z\in\Omega_i$,
    \begin{align}\label{eq:TDvector}
        \id^i\J(\Omega)(z)=\left(\id^{i\rightarrow j}\J(\Omega)(z)\right)_{j\in \M,j\neq i}\in\R^{M-1}.
    \end{align}
    Combining this for all possible $z\in\Omega$, we get the multi-material vector-valued topological derivative
    \begin{align}\label{eq:TDmulti}
        \id\J(\Omega)(z)=\sum_{i\in\M}\chi_{\Omega_i}(z)\id^i\J(\Omega)(z).
    \end{align}
\end{definition}
We call a material configuration $\Omega$ a minimizer of $\J$ with respect to topological perturbations, if for all points $z\in\Omega$, all possible material changes would lead to an increase of the functional $\J$.
\begin{definition}\label{def:TDoptimality}
A configuration $\Omega$ is called locally minimal with respect to topological perturbations if
    \begin{align}\label{eq:TDoptimality}
        \id\J(\Omega)(z)>0
    \end{align}
    for all $z\in\Omega$.
\end{definition}
\subsection{Multi-material level set description}
We use a vector-valued continuous level set function $\psi: D \rightarrow \mathbb R^{M-1}$ to describe a material configuration consisting of $M$ subdomains $\Omega_1, \dots, \Omega_M$ as it was done in \cite{Gangl_Multi}. We divide $\R^{M-1}$ into $M$ open sectors $S_i$ and use the relation
\begin{align}\label{eq:levelset}
    z\in\Omega_i\Leftrightarrow \psi(z)\in S_i
\end{align}
to assign a material to every point $z$. 
\begin{remark}
    For the commonly used two-material case (i.e. in or out), this corresponds to a scalar-valued level set function with sectors $S_1=(0,\infty),S_2=(-\infty,0)$ which leads to 
    \begin{align}
        z\in \Omega_1\Leftrightarrow\psi(z)>0, \quad z\in\Omega_2\Leftrightarrow\psi(z)<0.
    \end{align}
\end{remark}
Similarly as in \cite{Yamada_Multi}, we choose the sectors
\begin{align}
    S_i:=\{x\in\R^{M-1}:x\cdot V_i>x\cdot V_j, j\in \M\setminus\{i\}\}
\end{align}
where $V_i$ are the vertices of the $M-1$ dimensional unit simplex
\begin{align}\label{eq:Simplex}
    v_{i,n}=\begin{cases}
        -\sqrt{\frac{M}{(M-1)(M-n)(M-(n-1))}},&\quad n<i\\
        \sqrt{\frac{M(M-n)}{(M-1)(M-(n-1))}},&\quad n=i\\
        0,&\quad n>i,
    \end{cases}
\end{align}
where $v_{i,n}$ is the $n$-th entry of $V_i\in\R^{M-1}$ for $i=1,...,M.$
All sectors are congruent and share a hyperface with each other sector. Therefore, we do not introduce a bias by our material representation and enable direct changes between any two materials. A sector $S_i$ is bounded by $M-1$ half-hyperplanes with normals
\begin{align}
    n^{i\rightarrow j}=\frac{V_j-V_i}{\|V_j-V_i\|_2}.
\end{align}
This yields an alternative definition of the sectors
\begin{align}\label{eq:sectors}
    S_i=\{x\in\R^{M-1}: N_ix>0\},\quad N_i:=\left(\left(n^{j\rightarrow i}\right)^T\right)_{j\in \M,j\neq i} \in \mathbb R^{M-1 \times M-1},
\end{align}
where the row vectors of the matrix $N_i$ are all normals pointing into sector $S_i$.
\begin{example}
    In our application with $M=4$ materials, the points given by \eqref{eq:Simplex} are the vertices of a regular tetrahedron
    \begin{align}
        V_1=\begin{pmatrix}
            1\\0\\0
        \end{pmatrix}, V_2=\begin{pmatrix}
            -\tfrac{1}{3}\\\sqrt{\tfrac{8}{9}}\\0
        \end{pmatrix}, V_3=\begin{pmatrix}
            -\tfrac{1}{3}\\-\sqrt{\tfrac{2}{9}}\\\sqrt{\tfrac{4}{6}}
        \end{pmatrix}, V_4=\begin{pmatrix}
            -\tfrac{1}{3}\\-\sqrt{\tfrac{2}{9}}\\-\sqrt{\tfrac{4}{6}}
        \end{pmatrix}.
    \end{align}
\end{example}
Combining the optimality criterion from Definition \ref{def:TDoptimality} and the definition of the sectors \eqref{eq:sectors}, we can state a sufficient optimality criterion via the level set function.
\begin{lemma}\label{lem:Optimality}
    Let $\psi_{\Omega}$ be a vector-valued level set function describing the configuration $\Omega$. If
    \begin{align}\label{eq:Optimality}
        \psi_{\Omega}(z)=c\sum_{i\in\M}\chi_{\Omega_i}(z)N_i^{-1}\id^i\J(\Omega)(z)
    \end{align}
    for a positive constant $c>0$ and all $z\in\Omega$, then $\Omega$ is locally optimal with respect to topological perturbations.
\end{lemma}
\begin{proof}
    Let $z\in\Omega_i$ for an arbitrary $i\in\M$. By the fundamental property of the level set function \eqref{eq:levelset} and the definition of the sectors \eqref{eq:sectors}, we get
    \begin{align}
        z\in\Omega_i\Leftrightarrow\psi_\Omega(z)\in S_i\Leftrightarrow N_i\psi_\Omega(z)>0\Leftrightarrow N_i\left(c\sum_{k\in\M}\chi_{\Omega_k}(z)N_k^{-1}\id^k\J(\Omega)(z)\right)>0\Leftrightarrow \id^i\J(\Omega)(z)>0,
    \end{align}
    since $z\in\Omega_i$ and $c>0$. Since this holds true for all $i\in\M$, we get
    \begin{align}
        \id^i\J(\Omega)(z)>0 \forall i\in\M\Leftrightarrow\sum_{i\in\M}\chi_{\Omega_i}\id^i\J(\Omega),(z)>0\Leftrightarrow\id\J(\Omega)(z)>0
    \end{align}
    which is the optimality criterion \eqref{eq:TDoptimality}.
\end{proof}

We call the object
\begin{align}\label{eq:TDgeneralized}
    g_\Omega:=\sum_{i\in\M}\chi_{\Omega_i}N_i^{-1}\id^i\J(\Omega)
\end{align}
the generalized topological derivative.
\subsection{Level set algorithm}
We recall the level set algorithm introduced in \cite{Amstutz_Levelset}, generalized to the multi-material setting in \cite{Gangl_Multi}. We restrict the level set functions to the multi-dimensional $L_2$ unit sphere $\mathcal{S}:=\{v\in (L_2(D))^{M-1}:\|v\|_{(L_2(D))^{M-1}}=1\},$ since $k\psi$ describes the same configuration as $\psi$ for all $k>0.$ By choosing $c=\frac{\|\psi_\Omega\|_{(L_2(D))^{M-1}}}{\|g_\Omega\|_{(L_2(D))^{M-1}}}$ in the optimality condition \eqref{eq:Optimality} of Lemma \ref{lem:Optimality}, one gets that a design $\Omega$ is optimal if the $L_2$-angle between level set function $\psi_\Omega$ and generalized topological $g_\Omega$ \eqref{eq:TDgeneralized} is zero:
\begin{align}\label{eq:LStheta}
    \theta=\arccos\frac{\left(\psi_\Omega,g_{\Omega}\right)_{(L_2(D))^{M-1}}}{\|g_{\Omega}\|_{(L_2(D))^{M-1}}}=0.
\end{align}
 This is used as stopping criterion in Algorithm \ref{alg:Levelset} with a tolerance $\varepsilon>0$. The update step in Algorithm \ref{alg:Levelset} is a spherical linear interpolation between $\psi_\Omega$ and $g_\Omega$ with step size $s\in(0,1]$
\begin{align}\label{eq:LSupdate}
    \psi_\Omega=\frac{1}{\sin\theta}\left(\sin((1-s)\theta)\psi_\Omega+\sin(s\theta)\frac{g_{\Omega}}{\|g_{\Omega}\|_{L_2(D)^{M-1}}}\right),
\end{align}
which maintains the normalization of the level set function $\|\psi_\Omega\|_{L_2(D)^{M-1}}=1$.
As shown in \cite{Gangl_Multi}, this algorithm evolves along a descent direction, leading to an optimal design according to Definition \ref{def:TDoptimality}.
\begin{remark}
    The level set function $\psi$ has to be continuous, especially at the interfaces $\partial\Omega_i$. In order to guarantee the continuity also after the update \eqref{eq:LSupdate}, one has has to regularize the generalized topological derivative \eqref{eq:TDgeneralized}, which may be discontinuous. We do this by solving the PDE
    \begin{align}\label{eq:smoothing}
        -\rho\Delta \tilde{g}+\tilde{g}=g_\Omega\;\text{ on }D
    \end{align}
    for $\rho>0$ and update the level set function with $\tilde{g}\in (H^1(D))^{M-1}$. This is similar to a sensitivity filtering, e.g., by spatial averaging, used in density based design optimization \cite{Lazarov_Filter}, and introduces a lengthscale into the optimization.
\end{remark}
\begin{algorithm}
\caption{Level Set Algorithm for Topology Optimization by Topological Derivative.}
    \label{alg:Levelset}
\begin{algorithmic}
    \State\textbf{Choose} $\psi_0\in\mathcal{S}, k=0,k_\mathrm{max}<\infty, \varepsilon>0, 0<s_\mathrm{min}<s_\mathrm{max}\le 1,\gamma\in(0,1), \delta\ge1$
    \State \textbf{Evaluate} $\J(\Omega_{\psi_0})$
    \For{$k=0,...,k_\mathrm{max}$}
    \State \textbf{Compute} $g_{k}=g_{\Omega_{\psi_k}}$ 
    \If{$\theta_k=\arccos\frac{\left(\psi_k,g_{k}\right)_{(L_2(D))^{M-1}}}{\|g_{k}\|_{(L_2(D))^{M-1}}}<\varepsilon$}
    \State\textbf{break}
    \EndIf
    \State $s=s_\mathrm{max}$
    \While{$s> s_\mathrm{min}$}
    \State $\psi_{k+1}=\frac{1}{\sin\theta_k}\left(\sin((1-s)\theta_k)\psi_k+\sin(s\theta_k)\frac{g_{k}}{\|g_{k}\|_{L_2(D)^{M-1}}}\right)$
    \If {$\J(\Omega_{\psi_{k+1}})<\J(\Omega_{\psi_k})$} 
    \State{$s=\min\{s_\mathrm{max},\delta s\}$}
    \State\textbf{break}
    \Else
    \State $s=\max\{s_\mathrm{min},\gamma s\}$
    \EndIf
    \EndWhile
    \EndFor
\end{algorithmic}
\end{algorithm}
\subsection{Volume constraint}
In most applications, it may be useful to control the usage of some materials, e.g., in our application, the amount of permanent-magnet material. We introduce the volume constraint
\begin{align}
    \V(\Omega):= \sum_{i\in\M_\V}|\Omega_i| \le V^*,\quad \M_\V\subsetneq\M,
\end{align}
where $\M_\V$ is the index set of the materials which we want to constrain.
The topological derivative of $\V$ is given by
\begin{align}
    \id^{i\to j}\V(\Omega)=-\begin{cases}
        1&\text{ if }i\in\M_\V\\0&\text{ else}
    \end{cases}+\begin{cases}
        1&\text{ if }j\in\M_\V\\0&\text{ else}
    \end{cases}.
\end{align}
A common way is to add the volume constraint by a linear penalization to the objective 
\begin{align}\label{eq:Vollinear}
\J(\Omega)+\ell\V(\Omega).    
\end{align}
This balances the performance $\J(\Omega)$ and the material usage $\V(\Omega)$ according to the weight $\ell,$ which has to be determined experimentally. A more sophisticated approach is to include the constraint by an Augmented Lagrangian \cite{Nocedal_Wright} which, however, also requires tuning of the involved parameters.

In \cite{Beck_Deflation}, the authors present a very simple method to incorporate a volume constraint by shifting the level set function accordingly after each update step. In our application this is would lead to an infinite loop of derivative-driven updates and volume corrections. Therefore, we include this shift in the update step \eqref{eq:LSupdate}. We propose an approach based on the simple linear penalization \eqref{eq:Vollinear} with an adaptive weight $\ell$, such that the constraint is strictly fulfilled in every iteration. We consider the weight $\ell$ explicitly by solving
\begin{align}
    \min_\Omega\J_\ell(\Omega),\quad\J_\ell(\Omega):=\J(\Omega)+\ell\V(\Omega),
\end{align}
and denote its generalized topological derivative by $g^{\J_\ell}_\Omega=g_\Omega^\J+\ell g_\Omega^\V$, where $g_\Omega^\J$ and $g_\Omega^\V$ are the generalized topological derivatives \eqref{eq:TDgeneralized} of $\J(\Omega)$ and $\V(\Omega)$, respectively. We start with $\ell=0$ and compute the next iterate $\Omega_{k+1}$ by \eqref{eq:LSupdate}. If $\V(\Omega_{k+1})>V^*$, we dismiss the step and try to find a weight $\ell$, such that the constraint is fulfilled. For a too big $\ell$ this may result in an overpenalization and reduce the performance $\J(\Omega)$. Therefore, we choose the smallest $\ell$ possible by bisection. Since the evaluation of the volume constraint $\V(\Omega)$ is usually very fast compared to the evaluation of $\J(\Omega),$ which may include solving PDEs,  this can be done without slowing down the overall optimization. The design $\Omega_{k+1}$ found by this procedure should still give a decrease in the objective $\J(\Omega_{k+1})<\J(\Omega_k)$, otherwise the step size $s$ gets decreased and a new weight $\ell$ has to be computed. The incorporation of this method into the level set optimization is described in Algorithm \ref{alg:Volume}.
\begin{remark}
Let us illustrate this for our application, where we distribute four materials $\Omega=(\Omega_f,\Omega_a,\Omega_{m_1},\Omega_{m_2})$ and aim to control the volume of the permanent magnets $\Omega_{m_1},\Omega_{m_2}$. For large $\ell$, the term from the volume constraint dominates in the update of the level set function\eqref{eq:LSupdate} 
\begin{align*}
    \frac{g_\Omega^{\J_\ell}}{\|g_\Omega^{\J_\ell}\|_{L_2(D)^{M-1}}}=\frac{g_\Omega^\J+\ell g_\Omega^\V}{\|g_\Omega^\J+\ell g_\Omega^\V\|_{L_2(D)^{M-1}}}\sim\frac{g_\Omega^\V}{\|g_\Omega^\V\|_{L_2(D)^{M-1}}}.
\end{align*}
It aims to minimize the magnet volume, therefore $g_\Omega^\V$ will "move $\psi$ away" from magnets, towards iron or air. If the original configuration fulfills the constraint $\V(\Omega_\psi)\le V^*$, it is guaranteed that $\V(\Omega_{\psi+tg_\Omega^\V})\le V^*$ for all $t\ge 0$. For this reason there exists for all $\psi$ with $\V(\Omega_\psi)\le V^*$ and $s\in (0,1)$ a weight $\ell<\infty$ such that the updated level set function $\psi^*=\frac{1}{\sin\theta}\left(\sin((1-s)\theta)\psi+\sin(s\theta)\frac{g_\Omega^{\J_\ell}}{\|g_\Omega^{\J_\ell}\|_{L_2(D)^{M-1}}}\right)$ fulfills the volume constraint $\V(\Omega_{\psi^*})\le V^*$.
\end{remark}
\begin{algorithm}
\caption{Level Set Algorithm for Topology Optimization by Topological Derivative with Explicit Volume Control.}
    \label{alg:Volume}
\begin{algorithmic}
    \State\textbf{Choose} $\psi_0\in\mathcal{S}$ with $\V(\Omega_{\psi_0})\le V^*, k=0,k_\mathrm{max}<\infty, \varepsilon>0, 0<s_\mathrm{min}<s_\mathrm{max}\le 1,\gamma\in(0,1), \delta\ge1,\Delta_\ell>0,\varepsilon_\ell>0$
    \State \textbf{Evaluate} $\J(\Omega_{\psi_0})$
    \For{$k=0,...,k_\mathrm{max}$}  
    \State \textbf{Compute} $g_{k}^\J=g_{\Omega_{\psi_k}}^\J,g_k^\V= g_{\Omega_{\psi_k}}^\V$
    \State \textbf{Define} $g_k(\ell)=g_k^\J+\ell g_k^\V$
    \State \textbf{Define} $\theta_k(\ell)=\arccos\frac{\left(\psi_k,g_{k}(\ell)\right)_{(L_2(D))^{M-1}}}{\|g_{k}(\ell)\|_{(L_2(D))^{M-1}}}$
    \State $\ell=0, s=s_\mathrm{max}$
    \While{$\theta_k(\ell)>\varepsilon$}
    \State $\psi_{k+1}(\ell)=\frac{1}{\sin\theta_k(\ell)}\left(\sin((1-s)\theta_k(\ell))\psi_k+\sin(s\theta_k(\ell))\frac{g_{k}(\ell)}{\|g_{k}(\ell)\|_{L_2(D)^{M-1}}}\right)$
    \While{$\V(\Omega_{\psi_{k+1}(\ell)})> V^*$}
    \State $\Delta_\ell=2\Delta_\ell, \ell=\ell+\Delta_\ell$
    \EndWhile
    \If{$\ell>0$}
    \While{$|\V(\Omega_{\psi_{k+1}(\ell)})-V^*|>\varepsilon_\ell$}
    \If{$\mathrm{sgn}(\V(\Omega_{\psi_{k+1}(\ell)})- V^*)=\mathrm{sgn}(\V(\Omega_{\psi_{k}})- V^*)$ }
    \State $\Delta_\ell=\Delta_\ell/2, \ell=\ell+\Delta_\ell$
    \Else
    \State $\Delta_\ell=-\Delta_\ell/2, \ell=\ell+\Delta_\ell$
    \EndIf
    \EndWhile
    \EndIf
    \State $\psi_{k+1}=\psi_{k+1}(\ell)$
    \If {$\J(\Omega_{\psi_{k+1}})<\J(\Omega_{\psi_k})$} 
    \State{$s=\min\{s_\mathrm{max},\delta s\}$}
    \State\textbf{break}
    \EndIf
    \If {$s = s_\mathrm{min}$} 
    \State\textbf{break}
    \EndIf
    \State $s=\max\{s_\mathrm{min},\gamma s\}$
    \EndWhile
    \EndFor
\end{algorithmic}
\end{algorithm}

\section{Topological derivative for torque maximization}\label{sec:T}
In this section, we are interested to optimize the average torque of the machine for a fixed volume of permanent-magnet material. Since the eddy currents have a negligible effect on the torque, we consider the magnetostatic problem, i.e., a simplification of \eqref{eq:magnetoquasistatics} 
\begin{align}\label{eq:magnetostatics}\begin{aligned}
    \widetilde{\cur}_xh_{\Omega(t)}(\cur_x u(t))&=j(t)&&\text{ in }D_\mathrm{all}\\
    u(t)&=0&&\text{ on }\Gamma_S\\
    u(t)|_{\Gamma_1}&=-u(t)|_{\Gamma_2},&&
\end{aligned}
\end{align}
for $t\in[0,T]$. This problem can be seen as a sequence of static problems which we have to solve for different rotor positions $\alpha$. We do this by the harmonic mortar approach \cite{Egger_Mortar}, where the authors consider \eqref{eq:magnetostatics} for rotor and stator separately and reinforce continuity by a Lagrange multiplier. We introduce the spaces $\mH=\{v:v|_{D_R}\in H^1(D_R), v|_{D_S}\in H^1(D_S), v|_{\Gamma_S}=0, v|_{\Gamma_1}=-v|_{\Gamma_2}\}, \mL=H^{-\tfrac{1}{2}}(\Gamma)$ and denote by ${e:(\mH\times\mL)\times(\mH\times\mL)\times[0,2\pi)\rightarrow\R^2}$ the magnetostatic state operator
\begin{align}\label{eq:mortar}
        e((u,\lambda),(v,\mu),\alpha)&=\begin{pmatrix}
       \int_{D_\mathrm{all}}h_\Omega(\cur u)\cdot\cur v-j(\alpha)\dx + \langle\lambda,(v|_{D_S}-v|_{D_R}\circ\rho_{-\alpha})\rangle_{\Gamma}\\
        \langle\mu,(u|_{D_S}-u|_{D_R}\circ\rho_{-\alpha})\rangle_{\Gamma}
         \end{pmatrix}.
\end{align}
Here, the three phase source current density is given by
\begin{align}
    j(\alpha)=\chi_{D_{A^+}}\hat{j}\sin(4\alpha+\phi_0)-\chi_{D_{B^-}}\hat{j}\sin(4\alpha+\phi_0+\tfrac{2\pi}{3})+\chi_{D_{C^+}}\hat{j}\sin(4\alpha+\phi_0+\tfrac{4\pi}{3})
\end{align}
with amplitude $\hat{j}$ and load angle $\phi_0$ as in Table \ref{tab:data}. Note that in contrast to \eqref{eq:magnetoquasistatics}, the vector potential $u$ is now used in Lagrangian coordinates for a fixed material configuration $\Omega$ and the rotation of the rotor is accounted for by the Lagrange multiplier. We discretize the Lagrange multiplier by harmonic basis functions which allow a very simple incorporation of the coordinate transformation $\rho_{-\alpha}.$ We use the formula from \cite{Egger_Mortar}, which is based on energy considerations, to compute the mechanical torque by
\begin{align}\label{eq:torque}
    T(u_\alpha,\lambda_\alpha,\alpha)=r_{\Gamma}\langle\lambda_\alpha,(\cur u_\alpha\cdot n_{\Gamma})\circ\rho_{-\alpha}\rangle_{\Gamma},
\end{align}
where $r_{\Gamma}$ is the outer radius of the rotor $D_R$ and $n_{\Gamma}$ the outer unit normal vector of $\Gamma$ and $u_\alpha,\lambda_\alpha$ are the solution of
\begin{align}
    e((u_\alpha,\lambda_\alpha),(v,\mu),\alpha)&=\begin{pmatrix}
        0\\0
    \end{pmatrix}
\end{align}
for all $(v,\mu)\in\mH\times\mL$ for a given rotor position $\alpha$.
\begin{remark}
    In our setting, the magnetic fields in the rotor and inherently also the torque are periodic by $T=\tfrac{1}{6N_\mathrm{pp}V_\mathrm{rot}}$, see \eqref{eq:periodicity}, corresponding to an angle of $\frac{360^\circ}{N_\mathrm{pp}6}=15^\circ$. To compute the average torque, we choose $N=11$ equidistant rotor positions
    \begin{align}\label{eq:angles}
        \alpha^n=15^\circ\frac{n}{N}, n=0,...,N-1.
    \end{align}
    We will denote by $(u^n,\lambda^n)=(u_{\alpha^n},\lambda_{\alpha^n})$ the solution of \eqref{eq:mortar} for rotor position $\alpha^n$ and abbreviate by a bold symbol
    \begin{align}
        \boldsymbol{v}=(v^0,...,v^{N-1})^T
    \end{align}
    the vector of this quantity for all time steps.
\end{remark}
\begin{remark}\label{eq:staticDiscrete}
    After discretization, system \eqref{eq:mortar} can be written as
    \begin{align}
        \begin{pmatrix}
            A_\Omega(u)+B_\alpha^T\lambda\\B_\alpha u\end{pmatrix}=\begin{pmatrix}
            j_\alpha\\0
        \end{pmatrix},
    \end{align}
    where $A_\Omega$ corresponds to the nonlinear magnetostatic operator, $B_\alpha$ realizes the coupling between stator and rotor rotated by $\alpha$ and $j_\alpha$ represents the source current. Note, that $B_\alpha$ is independent of $\Omega$, which will be useful in design optimization, where $\Omega$ changes. We apply Newton's method, to solve this system iteratively. The Newton update $(\delta u,\delta\lambda)^T$ is computed by
    \begin{align}
        \begin{pmatrix}
            A_\Omega'(u)&B_\alpha^T\\B_\alpha&0
        \end{pmatrix}\begin{pmatrix}
            \delta u\\\delta\lambda
        \end{pmatrix}=-\begin{pmatrix}
            A_\Omega(u)-j_\alpha+B_\alpha^T\lambda\\B_\alpha u
        \end{pmatrix}.
    \end{align}
\end{remark}

 Our first objective is to minimize the negative (i.e. maximize the positive) average torque
 \begin{align}\label{eq:CostTorque}
     \overline{\T}(\Omega)=-\frac{1}{N}\sum_{n=0}^{N-1}T(u^n(\Omega),\lambda^n(\Omega),\alpha^n),
 \end{align}
 where $(\boldu(\Omega),\boldl(\Omega))\in\mH^N\times\mL^N$ is the unique solution of \eqref{eq:mortar} for rotor positions $\boldsymbol{\alpha} = (\alpha^0, \dots, \alpha^{N-1})^T$ and design $\Omega$.

\subsection{Topological derivative for torque maximization}
The topological derivative of the torque in the context of nonlinear magnetostatics was first derived in \cite{AmstutzGangl2019}, based on the sensitivity analysis introduced in \cite{Amstutz_Sensitivity}. The evaluation of the topological derivative of \eqref{eq:CostTorque}, presented here, is based on the general framework from \cite{Gangl_Automated}. First we introduce the Lagrangian for the torque maximization problem
\begin{align}\label{eq:LagTorque}\begin{aligned}
    \mathcal{G_{\T}}(\Omega,(\boldsymbol{u},\boldsymbol{\lambda}),(\boldsymbol{p},\boldsymbol{\eta}))=\sum_{n=0}^{N-1}\bigg(-\frac{r_{\Gamma}}{N}\langle\lambda^n,(\cur u^n\cdot n_{\Gamma})\circ\rho_{-\alpha^n}\rangle_{\Gamma}+e((u^n,\lambda^n),(p^n,\eta^n),\alpha^n)\bigg).
\end{aligned}
\end{align}
The adjoint states $(\boldp,\bolde)\in\mH^N\times\mL^N$ are given via the adjoint equation, which is the derivative of the Lagrangian \eqref{eq:LagTorque} with respect to the states $(u^n,\lambda^n)$
\begin{align}\label{eq:adjointTorque}
\begin{aligned}
\nabla_{u^n,\lambda^n}e((u^n,\lambda^n),(p^n,\eta^n),\alpha^n)^T(v,\mu)=\frac{1}{N}r_{\Gamma}\begin{pmatrix}
    \langle\lambda^n,(\cur v\cdot n_{\Gamma})\circ\rho_{-\alpha}\rangle_{\Gamma}\\
    \langle\mu,(\cur u^n\cdot n_{\Gamma})\circ\rho_{-\alpha}\rangle_{\Gamma}
\end{pmatrix}
\end{aligned}
\end{align}
for all $(v,\mu)\in\mH\times\mL, n=0,...,N-1.$ Hereby, the derivative of $e$ reads
\begin{align}\label{eq:Operatoradjoint}
    \nabla_{u,\lambda}e((u,\lambda),(p,\eta),\alpha)^T(v,\mu)=\begin{pmatrix}
       \int_{D_\mathrm{all}}\mathrm{d}_bh_\Omega(\cur u)\cur p\cdot\cur v\dx + \langle \eta,(v|_{D_S}-v|_{D_R}\circ\rho_{-\alpha})\rangle_{\Gamma}\\
        \langle\mu,(p|_{D_S}-p|_{D_R}\circ\rho_{-\alpha})\rangle_{\Gamma}
         \end{pmatrix}.
\end{align}
The topological derivative is defined point wise and depends solely on the point evaluations of state and adjoint, which we abbreviate by $U^n_z:=\cur u^n(z), P^n_z:=\cur p^n(z).$ For changes from material $i\in\I$ to $j\in\I\setminus\{i\}$, the topological derivative formula is given by
\begin{align}
\label{eq:TDformula}
\begin{aligned}
    \id^{i\rightarrow j}&\overline{\T}(\Omega)(z)=f^{i\rightarrow j}_\T(\boldsymbol{U}_z,\boldsymbol{P}_z):=\\&\sum_{n=0}^{N-1}\bigg(\frac{1}{|\omega|}\int_{\R^2}(h_\omega^{i\rightarrow j}(\cur K_{U_z^n}+U_z^n)-h_\omega^{i\rightarrow j}(U_z^n)-\mathrm{d}_bh_\omega^{i\rightarrow j}(U_z^n)\cur K_{U_z^n})\cdot P_z^n\ \id\xi\\
    &+\frac{1}{|\omega|}\int_\omega(\mathrm{d}_bh_j(U_z^n)-\mathrm{d}_bh_i(U_z^n))\cur K_{U_z^n}\cdot P_z^n\ \id\xi+(h_j(U_z^n)-h_i(U_z^n))\cdot P_z^n\bigg),
\end{aligned}
\end{align}
where $K_{U}\in BL(\R^2):=\{v\in L_2^{loc}(\R^2):\nabla v\in L_2(\R^2)^2\}/\R$ solves the auxiliary exterior problem
\begin{align}
    \int_{\R^2}(h_\omega^{i\rightarrow j}(\cur K_{U}+U)-h_\omega^{i\rightarrow j}(U))\cdot\cur v\ \id\xi=-\int_\omega(h_j(U)-h_i(U))\cdot\cur v\ \id\xi
    \label{eq:exterior}
\end{align}
for all $v\in BL(\R^2)$ with  $U\in\R^2, \omega=B_1(0)$ and $ h_\omega^{i\rightarrow j}(b)=h_i(b)\chi_\omega+h_j(b)\chi_{\R^2\setminus\overline{\omega}}$. 
\subsection{Evaluation of the topological derivative}
For linear materials like air and permanent magnet, i.e $i,j\in\{a,m_1,m_2\}$, the solution $K_U$ of the exterior problem \eqref{eq:exterior} has an analytical representation with the nice property
\begin{align}
    \cur_\xi K_U|_\omega=-\frac{(\nu_j-\nu_i)U-(M_j-M_i)}{\nu_j+\nu_i},
\end{align}
see e.g. \cite{Gangl_Surface}, which leads to
\begin{align}\label{eq:TDlinear}
        f_\T^{i\rightarrow j}(\boldsymbol{U}_z,\boldsymbol{P}_z)=\sum_{n=0}^{N-1}2\frac{\nu_i}{\nu_j+\nu_i}\left((\nu_j-\nu_i)U^n_z-(M_j-M_i)\right)\cdot P_z^n,
\end{align}
with $\nu_a=\nu_0, \nu_{m_1}=\nu_{m_2}=\nu_m, M_a=(0,0)^T, M_{m_i}=B_Re_{\varphi_i}$ as in Table \ref{tab:data}. Note, that the first line of \eqref{eq:TDformula} vanishes if both material laws involved are linear. For nonlinear materials, we have to evaluate the topological derivative by \eqref{eq:TDformula} which is very expensive, since we have to solve the exterior problem \eqref{eq:exterior} for every $U^n_z\in\R^2, n=0,...,N-1$. We therefore aim to precompute sample values for selected $\boldsymbol{U}_z,\boldsymbol{P}_z\in\R^{2\times N}$ and to interpolate them online. First, we observe that the different rotor positions are independent of each other. It is thus sufficient to sample $U,P\in\R^2$. Next we see that the dependence on the adjoint $P^n_z$ is linear. Therefore, we can write
\begin{align}
        f^{i\rightarrow j}_\T(\boldsymbol{U}_z,\boldsymbol{P}_z)=\sum_{n=0}^{N-1}\begin{pmatrix}
            f_1^{i\rightarrow j}(U_z^n)\\f_2^{i\rightarrow j}(U_z^n)
        \end{pmatrix}\cdot P_z^n, 
        \label{eq:TDevaluate}
    \end{align}
    with
    \begin{align}\label{eq:TDsample}
    \begin{split}
        f_1^{i\rightarrow j}(U)=&\frac{1}{|\omega|}\int_{\R^2}(h_\omega^{i\rightarrow j}(\cur K_{U}+U)-h_\omega^{i\rightarrow j}(U)-\mathrm{d}_bh_\omega^{i\rightarrow j}(U)\cur K_{U})\cdot e_0\ \id\xi\\
    &+\frac{1}{|\omega|}\int_\omega(\mathrm{d}_bh_j(U)-\mathrm{d}_bh_i(U))\cur K_{U}\cdot e_0\ \id\xi+(h_j(U)-h_i(U))\cdot e_0,\\
    f_2^{i\rightarrow j}(U)=&\frac{1}{|\omega|}\int_{\R^2}(h_\omega^{i\rightarrow j}(\cur K_{U}+U_z^n)-h_\omega^{i\rightarrow j}(U)-\mathrm{d}_bh_\omega^{i\rightarrow j}(U)\cur K_{U})\cdot e_\frac{\pi}{2}\ \id\xi\\
    &+\frac{1}{|\omega|}\int_\omega(\mathrm{d}_bh_j(U)-\mathrm{d}_bh_i(U))\cur K_{U}\cdot e_\frac{\pi}{2}\ \id\xi+(h_j(U)-h_i(U))\cdot e_\frac{\pi}{2}.
    \end{split}
    \end{align}
Similarly as in \cite{AmstutzGangl2019}, one can check, that the exterior problem commutes with rotations of $U\in\R^2$, given by a rotation matrix $R\in\R^{2\times 2}$
\begin{align}
    \cur K_{RU}=R\cur K_U
\end{align}
which transfers to the topological derivative
\begin{align}\label{eq:TDtransform}
    \begin{pmatrix}
        f_1^{i\rightarrow j}(U)\\
        f_2^{i\rightarrow j}(U)
    \end{pmatrix}=R\begin{pmatrix}
            f_1^{i\rightarrow j}(R^TU)\\f_2^{i\rightarrow j}(R^TU)
        \end{pmatrix}.
\end{align}
We can use this property to reduce the sampling effort. For isotropic materials like iron and air, i.e. $i,j\in\{f,a\}$, we can reduce the precomputations to the modulus of $U$. We write $U=tR_\beta e_0$ and use
\begin{align}\label{eq:Rbeta}\begin{aligned}
    \begin{pmatrix}
        f_1^{i\rightarrow j}(U)\\
        f_2^{i\rightarrow j}(U)
    \end{pmatrix}=R_\beta\begin{pmatrix}
            f_1^{i\rightarrow j}(te_0)\\f_2^{i\rightarrow j}(te_0)
        \end{pmatrix},
\end{aligned}
\end{align}
to incorporate the angle $\beta$ in the online phase. The only case left is the change between iron and permanent magnets. Here, we can apply \eqref{eq:TDtransform} to the magnetization directions ${\varphi_i}$
\begin{align}\label{eq:Rphi}
    &\begin{pmatrix}
        f_1^{f\rightarrow m_i}(U)\\
        f_2^{f\rightarrow m_i}(U)
    \end{pmatrix}=R_{\varphi_i}\begin{pmatrix}
            f_1^{f\rightarrow m_0}(R_{\varphi_i}^TU)\\f_2^{f\rightarrow m_0}(R^T_{\varphi_i} U)
        \end{pmatrix},
\end{align}
where $m_0$ represents a permanent magnet with direction $e_0$, i.e. $h_{m_0}(b)=\nu_m(b-B_Re_0).$ We still need to sample $f_1^{f\rightarrow m_0}(U), f_2^{f\rightarrow m_0}(U)$ for $U\in\R^2$ but we are free to choose arbitrary magnetization directions ${\varphi_i}$. For changes from magnet to iron, one simply has to exchange the role of iron and magnet.

Summarizing, we have three different cases: If both materials are linear, i.e. $i,j\in\{a,m_1m_2\}$, we evaluate \eqref{eq:TDlinear} without sampling. Secondly, if both materials are isotropic, i.e. $i,j\in\{i,a\}$, we have to sample \eqref{eq:TDsample} for $t=|U|\in\R$ and apply \eqref{eq:Rbeta}. Finally, for iron and permanent magnets, we have to take samples of \eqref{eq:TDsample} for $U\in\R^2$ for a reference orientation of the permanent magnet and incorporate its real orientation online by \eqref{eq:Rphi}.

\section{Topological derivative for coupled electromagnetic-thermal problem}\label{sec:TCt}
The magnetic field in the machine induces eddy currents in conductors, of which the paths close at the front and back sides of the conductors. They decrease the efficiency of the machine since the energy dissipates. Further, the heat produced can lead to significant damage of the materials. In the iron parts, one can almost avoid the eddy currents by lamination, since this reduces the conductivity in axial direction heavily. It is possible to segment also the permanent magnets. However, due to the additional manufacturing complexity, they are solid throughout the whole axial length in many applications. Additionally, depending on their material composition, high temperatures will lead to a demagnetization of the permanent magnets. Here, we will consider non-segmented neodymium-iron-boron permanent magnets which start to demagnetize at $90^\circ\mathrm{C}$. It is therefore of high interest not only to control the total eddy-current losses, but to constrain the maximal temperature produced by eddy currents.
\subsection{Magnetoquasistatic}
Based on Faraday's law \eqref{eq:Maxwell}, one can compute the third component of the eddy-current density by
\begin{align}
    j_e=-\sigma\partial_tu,
\end{align} where $u$ is the third component of the magnetic vector potential. We could compute $j_e$ in a post processing based on the solutions $u$ of \eqref{eq:mortar}. As remarked in \cite{Hameyer_Eddy}, this would lead to a significant overrating of the eddy-current losses. Instead we focus on the magnetoquasistatic PDE \eqref{eq:magnetoquasistatics}, where eddy currents are directly included
\begin{align}\label{eq:magnetoquasistatics2}\begin{aligned}
    \sigma_{\Omega(t)}\partial_tu(t) + \widetilde{\cur}_xh_{\Omega(t)}(\cur_x u(t))&=j\quad&&\text{ in }D_\mathrm{all}\\
    u(t)&=0\quad&&\text{ on }\Gamma_S\\
    u(t)|_{\Gamma_1}&=-u(t)|_{\Gamma_2}&&\\
    \sigma_\Omega u(0)&=\sigma_\Omega(T)\quad&&\text{ on }D_\mathrm{all},
\end{aligned}
\end{align}
for $t\in[0,T], T$ as in \eqref{eq:periodicity}. With implicit Euler on $N$ equidistant timesteps
\begin{align}\label{eq:timesteps}
    t^n=\alpha^n\frac{60}{V_\mathrm{rot}360^\circ}=n\frac{T}{N}=n\tau,\quad \tau=\tfrac{60}{N_\mathrm{pp}V_\mathrm{rot}6N},
\end{align}
fitting to the rotor positions $\alpha^n$, we arrive at the time-discrete scheme using the harmonic mortar approach to account for the rotation: Find $(\boldu,\boldl)\in \mH^N\times\mL^N$ with
\begin{align}\label{eq:quasimortar}
    \begin{aligned}
    \begin{pmatrix}
        \int_{D\mathrm{all}}\sigma_\Omega\frac{u^n-u^{n-1}}{\tau}v\dx\\0
    \end{pmatrix}+e((u^n,\lambda^n),(v^n,\mu^n),\alpha^n)&=\begin{pmatrix}
        0\\0
    \end{pmatrix}\\
        \sigma_\Omega u^{-1}&=\sigma_\Omega u^{N-1}
    \end{aligned}
\end{align}
for all test functions $(\boldv,\boldm)\in \mH^N\times\mL^N, n=0,...,N-1$ with the magnetostatic operator $e$ as in \eqref{eq:mortar}. 
\begin{remark}\label{rem:quasiDiscrete}
   Similarly as in Remark \ref{eq:staticDiscrete} we state the fully discretized version of \eqref{eq:quasimortar}
   \begin{align}\begin{aligned}
       \tfrac{\sigma_m}{\tau}M_\Omega(u^n-u^{n-1})+A_\Omega(u^n)+B_n^T\lambda^n&=j_n\\
       B_nu^n&=0\\
       M_\Omega u^{-1}&=M_\Omega u^{N-1},
   \end{aligned}
   \end{align}
   where $M_\Omega$ is the mass matrix supported on $\Omega_{m_1}\cup\Omega_{m_2},$ $A_\Omega$ is coming from the nonlinear magnetostatic operator, $B_n$ realizes the rotation by angle $\alpha^n$ and $j_n$ represents the source current density. 
   We solve this system all-in-one using Newton's method. The residuum $r$ reads
   \begin{align}
       r(\boldu,\boldl)=\begin{pmatrix}
            \tfrac{\sigma_m}{\tau}M_\Omega(u^0-u^{N-1})+\tilde{A}_{\Omega,n}(u^0,\lambda^0)\\
            \tfrac{\sigma_m}{\tau}M_\Omega(u^1-u^{0})+\tilde{A}_{\Omega,n}(u^1,\lambda^1)\\
            \vdots\\
           \tfrac{\sigma_m}{\tau}M_\Omega(u^{N-1}-u^{N-2})+\tilde{A}_{\Omega,n}(u^{N-1},\lambda^{N-1})
        \end{pmatrix}.
   \end{align}
   The linearized system, which we have to solve in each Newton step is
    \begin{align}
       \tilde{K}_{\Omega,n}(\boldu)\boldsymbol{\delta}=-r(\boldu,\boldl),
    \end{align}
    where $\boldsymbol{\delta}=(\delta u^0,\delta\lambda^0,...,\delta u^{N-1},\delta\lambda^{N-1})^T$ is the update of $(u^0,\lambda^0,...,u^{N-1},\lambda^{N-1})^T$ and
    \begin{align}\label{eq:K_allinone}
    \tilde{K}_{\Omega,n}(\boldu)&=\begin{pmatrix}
            \tilde{M}_\Omega+\tilde{A}_{\Omega,n}'(u^0)&0&\dotsi&-\tilde{M_\Omega}\\
            -\tilde{M_\Omega}&\tilde{M_\Omega}+\tilde{A}_{\Omega,n}'(u^1)&\dotsi&0\\
            \vdots&\ddots&\ddots&\vdots\\
            0&\dotsi&-\tilde{M_\Omega}&\tilde{M_\Omega}+\tilde{A}_{\Omega,n}'(u^{N-1})\\
        \end{pmatrix},
    \end{align}
    where the block components are given by
    \begin{align}
        \tilde{M}_\Omega&=\begin{pmatrix}
            \tfrac{\sigma_m}{\tau}M_\Omega&0\\0&0
        \end{pmatrix},\tilde{A}_{\Omega,n}(u^n,\lambda^n)=\begin{pmatrix}
            A_\Omega(u^n)-j_n+B_n^T\lambda^n\\B_nu^n
        \end{pmatrix},\tilde{A}_{\Omega,n}'(u^n)=\begin{pmatrix}
            A_\Omega'(u^n)&B_n^T\\B_n&0
        \end{pmatrix}.
    \end{align}
    The matrix $M_\Omega$ is independent of the rotation angle, whereas the matrices $B_n$ are independent of the design $\Omega$ and can be pre-assembled for all rotor positions $\alpha^n$. For large $N$ and fine spatial discretizations, this system grows quickly and one has to apply special techniques to solve it efficiently. However, in our applications this is not the case. We speed up the computations by taking solutions of the magnetostatic problem \eqref{eq:mortar} as initial guess for Newton's method, which ensures convergence after a few iterations. 
\end{remark}
\begin{remark}
    Instead of considering periodic boundary conditions in \eqref{eq:magnetoquasistatics2}, one could run an initial value problem sufficiently long with a solution of the magnetostatic problem \eqref{eq:magnetostatics} as initial condition. We observed, on the one hand, that the torque quickly reaches its steady state. On the other hand, the eddy-current losses need far longer. Since we are here especially interested in these losses, we choose the first approach yielding the steady state by construction of the problem.
\end{remark}
\begin{remark}
    For the unique solvability of the time periodic magnetoquasistatic problem \eqref{eq:magnetoquasistatics}, we refer to \cite{Cesarano_SpaceTime}.
\end{remark}
\subsection{Topological derivative for magnetoquasistatics}
We are still interested in maximizing the average torque given similarly as in \eqref{eq:torque}:
\begin{align}\label{eq:CosttorqueED}
    \overline{\T_\mathrm{ed}}(\Omega)=-\frac{1}{N}\sum_{n=0}^{N-1}T(u^n(\Omega),\lambda^n(\Omega),\alpha^n),
\end{align}
where $(\boldu(\Omega),\boldl(\Omega))\in \mH^N\times\mL^N$ are now the solution of the magnetoquasistatic equation \eqref{eq:quasimortar} for a design $\Omega$. By the subscript $\T_\mathrm{ed}$, we denote that the torque is computed considering eddy currents. The topological derivative for this problem was derived in \cite{Krenn_Spacetime} considering a linear material behavior. These results generalize straightforwardly to the nonlinear case. We will see that the topological derivative formula is highly similar to the one of the magnetostatic case \eqref{eq:TDformula}, considered in Section \ref{sec:T}. First, we introduce the Lagrangian of \eqref{eq:CosttorqueED} based on the Lagrangian of the static problem $\mathcal{G}_\T$ \eqref{eq:LagTorque}
\begin{align}\label{eq:LagTorqueED}\begin{aligned}
    \mathcal{G}_{\T_\mathrm{ed}}&(\Omega,(\boldu,\boldl),(\boldp,\bolde))=\mathcal{G}_{\T}(\Omega,(\boldu,\boldl),(\boldp,\bolde))+\sum_{n=0}^{N-1}\int_{D_\mathrm{all}}\frac{\sigma_\Omega}{\tau}(u^n-u^{n-1})p^n\dx.
\end{aligned}
\end{align}
Taking the derivative with respect to the states $(u^n,\lambda^n)$ yields the adjoint states $(p^n,\eta^n)\in\mH\times\mL$, which are the solutions of the adjoint equation
\begin{align}\label{eq:adjointTorqueED}
\begin{aligned}
\begin{pmatrix}
     \int_{D\mathrm{all}}\frac{\sigma_\Omega}{\tau}(p^n-p^{n+1})v\dx\\0
\end{pmatrix}+\nabla_{u^n,\lambda^n}e(&(u^n,\lambda^n),(p^n,\eta^n),\alpha^n)(v,\mu)\\
   &=\frac{1}{N}r_{\Gamma}\begin{pmatrix}
    \langle\lambda^n,(\cur v\cdot n_{\Gamma})\circ\rho_{-\alpha}\rangle_{\Gamma}\\
    \langle\mu,(\cur u^n\cdot n_{\Gamma})\circ\rho_{-\alpha}\rangle_{\Gamma}
\end{pmatrix}\\
        \sigma_\Omega p^N&=\sigma_\Omega p^0,
\end{aligned}
\end{align}
for all $(v,\mu)\in\mH\times\mL, n=0,...,N-1$, with $\nabla_{u,\lambda}e$ as in \eqref{eq:Operatoradjoint}.
\begin{remark}
    We observe in the first term of \eqref{eq:adjointTorqueED}, that the structure of the adjoint equation is now "backwards-in-time". We need $p^{n+1}$ in order to solve for $p^n$. However, the periodic boundary condition of the forward problem \eqref{eq:quasimortar} was transferred too. Therefore, we will use again the all-in-one approach, to solve this system. Referring to the notation of Remark \ref{rem:quasiDiscrete}, the discrete system for the adjoint equation \eqref{eq:adjointTorqueED} reads
    \begin{align}\begin{aligned}
            \tilde{K}^T_{\Omega,n}(\boldu)\begin{pmatrix}
            (p^0,\eta^0)^T\\(p^1,\eta^1)^T\\\vdots\\(p^{N-1},\eta^{N-1})^T
        \end{pmatrix}=\frac{1}{N}\begin{pmatrix}
        \nabla_{u,\lambda}T(u^0,\lambda^0)\\\nabla_{u,\lambda}T(u^1,\lambda^1)\\\vdots\\\nabla_{u,\lambda}T(u^{N-1},\lambda^{N-1})
       \end{pmatrix}.
    \end{aligned}
    \end{align}
    The system matrix is the transposed Newton matrix from Remark \ref{rem:quasiDiscrete}. The transposition corresponds to the backwards-in-time structure.
\end{remark}
For a point $z\in\Omega_i$, the topological derivative of \eqref{eq:CosttorqueED} from material $i\in\I$ to material $j\in\I\setminus\{i\}$ is given by
\begin{align}\label{eq:TDformulaED}
    \id^{i\rightarrow j}\overline{\T_\mathrm{ed}}(\Omega)(z)=f^{i\rightarrow j}_{\T_\mathrm{ed}}(\boldsymbol{U}_z,\boldsymbol{P}_z,\boldsymbol{U}_z,\boldsymbol{P}_z):=f^{i\rightarrow j}_\T(\boldsymbol{U}_z,\boldsymbol{P}_z)+\sum_{n=0}^{N-1}\frac{\sigma_j-\sigma_i}{\tau}(u^n_z-u^{n-1}_z)p^n_z,
\end{align}
where one can reuse $f^{i\rightarrow j}_\T$ from \eqref{eq:TDformula} and insert $U^n_z=\cur u^n(z), P_z^n=\cur p^n(z)$, with $u^n$ the first component of the solutions of the magnetoquasistatic problem \eqref{eq:quasimortar} and $p^n$ the corresponding first component of the adjoint states from \eqref{eq:adjointTorqueED} for $n=0,...,N-1$. The only major change is the last term coming from the eddy currents with $u^n_z=u^n(z), p^n_z=p^n(z)$.
\begin{remark}
    For a more rigorous derivation of the topological derivative formula, we refer to \cite{Krenn_Spacetime}, where the authors considered a transient linear heat equation on moving domains with an application to a synchronous reluctance machine. They considered the equation in Eulerian coordinates, which lead to an additional convection term which we do not need to account for, since \eqref{eq:quasimortar} is stated in Lagrangian coordinates. The nonlinearity of the materials is independent of time or movement, and is incorporated straightforwardly. 
\end{remark}
\subsection{Modeling eddy-current losses}
We compute the eddy currents by finite differences
\begin{align}
    j_e^n=-\sigma_\Omega\frac{u^n-u^{n-1}}{\tau}, n=0,...,N-1.
\end{align}
In \cite{Hameyer_Eddy}, the authors point out the need to adapt the eddy currents to restore their physical property. Since they form a closed loop in axial direction, the integral over their 2d surface $\Omega_{m_1},\Omega_{m_2}$ has to be zero. We correct this by subtracting their spatial average
\begin{align}\label{eq:eddySpatialAverage}
    \widetilde{J^n_{m_i}}=\chi_{\Omega_{m_i}}\left(j_e^n-\overline{j_e^n}^{m_i}\right)=\chi_{\Omega_{m_i}}\left(j_e^n-\frac{1}{|\Omega_{m_i}|}\int_{\Omega_{m_i}}j_e^n\dx\right).
\end{align}
Finally, the temporally averaged loss density in $\Omega_{m_i}$ is given by $P^\mathrm{ed}_{m_i}=\tfrac{\sigma_m}{N}\sum_{n=0}^{N-1}\widetilde{J^n_{m_i}}^2$, where $\sigma_m$ is the conductivity of the permanent magnets. We combine this for all domains to
\begin{align}\label{eq:POmega}
    P^\mathrm{ed}_\Omega=\sum_{i\in\I}P_i^\mathrm{ed}\chi_{\Omega_i}
\end{align}
where $P_f^\mathrm{ed}=P_a^\mathrm{ed}=0$ and for the permanent magnets
\begin{align}\label{eq:eddylosses}
    P^\mathrm{ed}_{m_i}(\boldu)=\frac{\sigma_m}{N\tau^2}\sum_{n=0}^{N-1}\left(u^n-u^{n-1}-\frac{1}{|\Omega_{m_i}|}\int_{\Omega_{m_i}}u^n-u^{n-1}\dx\right)^2,\; i=1,2.
\end{align}
\begin{remark}
    If the magnets are segmented in circumferential direction, i.e., either $\Omega_{m_1}$ or $\Omega_{m_2}$ are not connected, one has to consider the average for each segment separately. Similarly, one needs to consider the full area if they are connected electrically. The recognition and handling of these components throughout the free-form design optimization introduce substantial additional difficulties, which are not addressed here. We perform our computations with the modeling assumption that the magnets consist of two distinct connected parts.
\end{remark}
\begin{remark}
    The overall average eddy-current loss power $\mathcal{P}_\mathrm{ed}(\Omega)$ can be computed by
    \begin{align}\label{eq:overallEddy}
        \mathcal{P}_\mathrm{ed}(\Omega)=\int_{D}P^\mathrm{ed}_\Omega(\boldu(\Omega))\dx,
    \end{align}
    where $\boldu(\Omega)$ is the solution of the magnetoquasistatic problem \eqref{eq:quasimortar} for the configuration $\Omega$.
\end{remark}
We compute the resulting temperature distribution $\vartheta$ by solving a static heat equation
\begin{align}\label{eq:heat}
\begin{aligned}
    -\mathrm{div}\lambda_\Omega\nabla\vartheta&=P_\mathrm{ed}(u^0,...,u^{N-1})&&\text{ in } D\\
    \lambda_\Omega\nabla\vartheta\cdot n&=\beta_{\mathrm{SH}}(\vartheta_0-\vartheta)&&\text{ on }\Gamma_\mathrm{SH}\\
    \lambda_\Omega\nabla\vartheta\cdot n&=\beta_{\mathrm{AG}}(\vartheta_0-\vartheta)&&\text{ on }\Gamma_\mathrm{R}.\\
\end{aligned}
\end{align}
The material coefficient $\lambda_\Omega=\sum_{i\in\I}\chi_{\Omega_i}\lambda_i$ is defined piecewise. The values for the thermal conductivities $\lambda_f,\lambda_a,\lambda_{m_1},\lambda_{m_2}$, the heat transfer coefficients $\beta_\mathrm{SH},\beta_\mathrm{AG}$ and the ambient temperature $\vartheta_0$ are given in Table \ref{tab:data}.
\subsection{Temperature constraint}
We want to constrain the temperature $\vartheta$ in each point in the magnets by a limit temperature $T^*$
\begin{align}\label{eq:constraintTempPW}
    \vartheta(x)\le T^*\text{ a.e. on }\Omega_{m_1}\cup\Omega_{m_2}.
\end{align}
We translate this constraint to a functional,
\begin{align}\label{eq:constraintTranslation}\begin{aligned}
    \vartheta(x)\le T^*&\Leftrightarrow\frac{\vartheta(x)}{T^*}\le1\Leftrightarrow\max\left\{\frac{\vartheta(x)}{T^*},1\right\}-1=0\text{ a.e. on }\Omega_{m_1}\cup\Omega_{m_2}\\&\Leftrightarrow\int_{\Omega_{m_1}\cup\Omega_{m_2}}\max\left\{\frac{\vartheta(x)}{T^*},1\right\}-1\dx=0,
\end{aligned}
\end{align} 
which is zero if and only if the constraint is fulfilled. In order to be able to compute derivatives, we square the maximum function and result at the quadratic penalty function used in \cite{AndradeNovotnyLaurain2024}:
\begin{align}\label{eq:Ctemp}
    \mathcal{C}_\mathrm{th}(\Omega)=\int_Dc_\Omega\Theta_{T^*}(\vartheta(\Omega))\dx,\quad\Theta_S(s)=\left(\max\left\{1,\frac{s}{S}\right\}-1\right)^2 \; \text{ for } S>0,
\end{align}
where $c_\Omega=\sum_{i\in\I}c_i\chi_{\Omega_i}, c_f=c_a=0, c_{m_1}=c_{m_2}=1$ and $(\boldu(\Omega),\boldl(\Omega),\vartheta(\Omega))\in \mH^n\times\mL^n\times H^1(D)$ is the solution of
\begin{align}\label{eq:magnetothermal}\begin{aligned}
   \begin{pmatrix}
        \int_{D\mathrm{all}}\sigma_\Omega\frac{u^n-u^{n-1}}{\tau}v\dx\\0
    \end{pmatrix}+e((u^n,\lambda^n),(v^n,\mu^n),\alpha^n)&=\begin{pmatrix}
        0\\0
    \end{pmatrix}\\
        \sigma_\Omega u^{-1}&=\sigma_\Omega u^{N-1}\\
        \int_D\lambda_\Omega\nabla\vartheta\cdot\nabla\varphi\dx+\int_{\Gamma_\mathrm{SH}}\beta_\mathrm{SH}(\vartheta_0-\vartheta)\varphi\ds+\int_{\Gamma_R}\beta_\mathrm{AG}(\vartheta_0-\vartheta)\varphi\ds&=\int_DP^\mathrm{ed}_\Omega(\boldu)\varphi\dx
\end{aligned}
\end{align}
for all test functions $(v^n,\mu^n)\in\mH\times \mL, \varphi\in H^1(D),n=0,...,N-1$, where the magnetostatic operator $e$ is given in \eqref{eq:mortar}.
\begin{remark}
    The unique solvability of the Joule heating problem combined with magnetoquasistatics \eqref{eq:magnetothermal} was investigated in \cite{Lasarzik_Joule}. Combining this with the considerations about the harmonic mortar approach \cite{Egger_Mortar}, we can conclude the existence of a unique solution of problem \eqref{eq:magnetothermal}.
\end{remark}
\begin{remark}
    The functional \eqref{eq:Ctemp} is zero, if and only if the pointwise constraint \eqref{eq:constraintTempPW} is zero almost everywhere. We include the constraint in our design optimization problem, by adding the functional \eqref{eq:Ctemp} to the main objective with a suitable weight. If the weight is too small, the functional may be positive leading to violations of the constraint. Conversely, a too high choice may result in numerical instabilities in the optimization. However, the effort of tuning this single parameter is acceptable.
\end{remark}
\subsection{Toplogical derivative}
The structure of the shape function \eqref{eq:Ctemp} suits the framework presented in \cite{Gangl_Automated}. It turns out that the topological derivative almost separates to the magnetic and the thermal part with a simple term accounting for the coupling via the eddy-current losses. We can therefore reuse the topological derivative of the eddy-current problem $f^{i\rightarrow j}_{\T_\mathrm{ed}}$ \eqref{eq:TDformulaED} and rely on \cite{Amstutz_Sensitivity} for the topological derivative of the stationary heat equation. It remains to derive the adjoint system. Therefore, we introduce the Lagrangian of the electromagnetic-thermal coupled system
\begin{align}\label{eq:LagrTemp}\begin{aligned}
    \mathcal{G}_{\C_\mathrm{th}}&(\Omega,(\boldu,\boldl,\vartheta),(\boldp,\bolde,\phi))=\int_Dc_\Omega\Theta_{T^*}(\varphi)\dx+\int_D\lambda_\Omega\nabla\vartheta\cdot\nabla\phi\dx-\int_DP^\mathrm{ed}_\Omega(\boldu)\phi\dx\\&-\int_{\Gamma_\mathrm{SH}}\beta_\mathrm{SH}(\vartheta_0-\vartheta)\phi\ds-\int_{\Gamma_R}\beta_\mathrm{AG}(\vartheta_0-\vartheta)\phi\ds+\sum_{n=0}^{N-1}\bigg(\int_{D_\mathrm{all}}\frac{\sigma_\Omega}{\tau}(u^n-{u^{n-1}})p^n\dx +e((u^n,\lambda^n),(p^n,\eta^n),\alpha^n)\bigg),
\end{aligned}
\end{align}
with $e$ as in \eqref{eq:mortar}. We obtain the adjoint equation by differentiating the Lagrangian \eqref{eq:LagrTemp} with respect to the states $(\boldu,\boldl,\vartheta)$. The adjoint states $(\boldp,\bolde,\phi)\in\mH^N\times\mL^N\times H^1(D)$ are the solution of
\begin{align}\label{eq:adjointTemp}
    \begin{aligned}
        \int_D\lambda_\Omega\nabla\phi\cdot\nabla\varphi+\int_{\Gamma_\mathrm{SH}}\beta_\mathrm{SH}\phi\varphi\ds+\int_{\Gamma_R}\beta_\mathrm{AG}\phi\varphi\ds&=-\int_D c_\Omega\Theta_{T^*}'(\vartheta)\varphi\dx\\
        \begin{pmatrix}
        \int_{D\mathrm{all}}\frac{\sigma_\Omega}{\tau}(p^n-p^{n+1})v\dx\\0
\end{pmatrix}+\nabla_{u^n,\lambda^n}e((u^n,\lambda^n),(p^n,\eta^n),\alpha^n)(v,\mu) &=\begin{pmatrix}
    \int_D\partial_{u^n}P^\mathrm{ed}_\Omega(\boldu)(v)\phi\dx\\0
\end{pmatrix}\\
        \sigma_\Omega p^N&=\sigma_\Omega p^0,
    \end{aligned}
\end{align}
for all test functions $(v^n,\mu^n,\varphi)\in\mH\times\mL\times H^1(D), n=0,...,N-1,$ with $\nabla_{u,\lambda}e$ as in \eqref{eq:Operatoradjoint}. The derivative of the eddy-current losses in the permanent magnets \eqref{eq:eddylosses} is given by
\begin{align}
    \begin{aligned}
        \partial_{u^n}P^\mathrm{ed}_{m_i}(\boldu)(v)=\frac{\sigma_m}{N\tau^2}\left(-u^{n+1}+2u^n-u^{n-1}-\overline{(-u^{n+1}+2u^n-u^{n-1})}^{m_i}\dx\right)(v-\overline{v}^{m_i}),
    \end{aligned}
\end{align}
where $\overline{v}^{m_i}$ is the spatial average of a function $v$ over $\Omega_{m_i}$, see \eqref{eq:eddySpatialAverage}. The derivative of the penalty function is given by
\begin{align}\label{eq:zetadev}
    \Theta_S'(s)=\frac{2}{S}\left(\max\left\{1,\frac{s}{S}\right\}-1\right).
\end{align}
We observe again the temporally inverted structure of the adjoint equation. In the forward problem \eqref{eq:magnetothermal}, we computed the temperature $\vartheta$ based on the magnetic states $\boldu$. Here, we have to solve first for the thermal adjoint state $\phi$ to compute the magnetic adjoint states $\boldp$.
The topological derivative of $\C_\mathrm{th}(\Omega)$ from material $i\in\I$ to $j\in\I\setminus\{i\}$ in $z\in\Omega_i$ is given by
\begin{align}\label{eq:TDformulaTemp}\begin{aligned}
    \id^{i\rightarrow j}\C_\mathrm{th}(\Omega)(z)&=f^{i\rightarrow j}_{\T_\mathrm{ed}}(\boldsymbol{U}_z,\boldsymbol{P}_z,\boldsymbol{U}_z,\boldsymbol{P}_z)+2\lambda_i\frac{\lambda_j-\lambda_i}{\lambda_j+\lambda_i}\nabla\vartheta_z\cdot\nabla\phi_z+(c_j-c_i)\Theta_{T^*}(\vartheta_z)-(P^\mathrm{ed}_j(\boldsymbol{U}_z)-P^\mathrm{ed}_i(\boldsymbol{U}_z))\phi_z,
\end{aligned}
\end{align}
where $U^n_z=\cur u^n(z), P^n_z=\cur p^n(z), u^n_z=u^n(z), p^n_z=p^n(z), \vartheta_z=\vartheta(z),\phi_z=\phi(z)$ are now the point evaluations of the direct $(\boldu,\vartheta)$ and adjoint states $(\boldp,\phi)$ solving equations \eqref{eq:magnetothermal} and \eqref{eq:adjointTemp}, respectively and $f^{i\rightarrow j}_{\T_\mathrm{ed}}$ is defined in \eqref{eq:TDformulaED}.
\begin{remark}
    One obtains \eqref{eq:TDformulaTemp} by applying the framework presented in \cite{Gangl_Automated} to \eqref{eq:Ctemp}. The first term relates to the magnetoquasistatic equation, the second term is the well-known topological derivative of a linear diffusion equation, see e.g. \cite{Amstutz_Sensitivity}, and the latter two are the variations of the thermal cost functional and the heat source. In the spirit of \cite{Gangl_Automated}, we mention that the exterior problem, like \eqref{eq:exterior}, for the asymptotic expansion of the state $(\boldu,\boldl,\vartheta)$ decouples in the magnetic and thermal domain, since the coupling term via the eddy-current losses depends on the states themselves and not on their gradients. In the limit $\epsilon\searrow 0$, this term vanishes.
\end{remark}
\section{Mechanical constraints} \label{sec_mechanics}
In the numerical results, we will see that the designs obtained by optimizing their electromagnetic-thermal properties lack of mechanical stability. Although we fixed a predefined iron ring $D_\mathrm{RI}$ throughout the optimizations, the occurring forces would lead to a local material failure. Since increasing the width of this ring reduces the magnetic performance, we propose to handle this issue by imposing additional constraints to our design optimization problem. We constrain the pointwise Von Mises stresses in the rotor iron similarly as in \cite{Holley_Diss}. For the sake of readability, we will use standard notation of elasticity in this section, which overlaps with the symbols used in electromagnetics in Sections \ref{sec:T} and \ref{sec:TCt}. We write $u$ for the mechanical deformation, which is now the state variable, and denote its adjoint state by $p$. Mechanical stresses are denoted by $\sigma$, and $\nu$ is the Poisson ratio. We will not mix formulas from the mechanical domain with the electromagnetic-thermal considerations from above.
\subsection{Linearized elasticity}
The mechanical behavior of our machine loaded by centrifugal forces can be described by the equations of 2d linearized elasticity 
\begin{align}
\begin{aligned}\label{eq:elasticity}
    -\mathrm{div}\,\sigma_\Omega(u)&=\rho_\Omega V_\mathrm{rot}^2x&&\text{ in }D\cup D_\mathrm{RI}\\
    \sigma_\Omega(u)n&=0&&\text{ on }\Gamma_R\\
    u&=0&&\text{ on }\Gamma_\mathrm{SH}\\
    R_{\frac{\pi}{4}}u|_{\Gamma_1}&=u|_{\Gamma_2},
\end{aligned}
\end{align}
where the latter condition realizes the symmetry condition of the deformation on the radial boundaries. Further, we have 
\begin{align}\label{eq:sigmaOmega}
    \sigma_\Omega(u)=2\mu_\Omega\varepsilon(u)+\lambda_\Omega\mathrm{tr}(\varepsilon(u))\mathrm{I}, \quad\varepsilon(u)=\tfrac{1}{2}(\nabla u+\nabla u^T)
\end{align}
with piecewise defined Lamé parameters for plane stress $\mu_\Omega=\frac{E_\Omega}{2(1+\nu)}, \lambda_\Omega=\frac{E_\Omega\nu}{1-\nu^2}, E_\Omega=\sum_{i\in\I}E_i\chi_{\Omega_i}, \nu_\Omega=\sum_{i\in\I}\nu_i\chi_{\Omega_i}$ and the mass density $\rho_\Omega=\sum_{i\in\I}\rho_i$. We restrict ourselves to a uniform Poisson ratio $\nu=\frac{1}{3}$, which is a realistic value and leads to simplifications in the formulae for the topological derivative. We refer to \cite{Amstutz_vonMises} for the more general case. In order to be able to solve \eqref{eq:elasticity} also for designs including air, we assign an artificial Young's modulus $E_a$ to $\Omega_a$ with a finite contrast to iron $E_f/E_a<\infty$. This so-called ersatz material approach is widely used in mechanical design optimization. With this choice of material parameters, we can write
\begin{align}
    \sigma_\Omega(u)=E_\Omega\sigma(u),\quad \sigma(u)= \frac{3}{4}\varepsilon(u)+\frac{3}{8}\mathrm{tr}(\varepsilon(u))\mathrm{I}.
\end{align}
\begin{remark}
    In the production process, the iron sheets are usually punched and the permanent magnets glued in afterwards. The mechanical modeling of this situation is very challenging and would require to solve a contact problem. Especially in design optimization with changing material interfaces, this is infeasible. Therefore, we model the permanent magnets with a Young's modulus similar as air $E_{m_1}=E_{m_2}=E_a$ but assign a realistic mass density $\rho_{m_1}=\rho_{m_2}=\rho_m$ according to Table \ref{tab:data}, as it was also done in e.g. \cite{Lee_MTPAMech}. 
\end{remark}
\subsection{Von Mises stress constraint}
Similarly as in \cite{Holley_Diss}, we introduce a constraint to the local Von Mises stresses. The Von Mises stress or yield stress is a scalar quantity which is often considerd as a material failure criterion. For simplicity, we consider always the squared Von Mises stress
\begin{align}\label{eq:vonMises}
    s_\Omega(u):=E_\Omega^2s(u),\quad s(u)=\frac{1}{2}\mathbb{B}\sigma(u):\sigma(u),\quad \mathbb{B}=3\mathbb{I}-\mathrm{I}\otimes\mathrm{I}.
\end{align}
We are interested in limiting the Von Mises stress by some $S^*$
\begin{align}\label{eq:consVMPW}
    s_\Omega(u)\le (S^*)^2\Leftrightarrow s(u)\le M^*\text{ a.e. in }\Omega_f\cap\tilde{D},
\end{align}
with $M^*=\frac{(S^*)^2}{E_f^2}$. As pointed out in \cite{Amstutz_vonMises}, the stresses next to the Dirichlet boundary $\Gamma_\mathrm{SH}$ will be artificially high. Therefore, we consider the stress constraint only inside $\tilde{D}\subset D\cup D_\mathrm{RI}$ which has a sufficiently large distance to $\Gamma_\mathrm{SH}$. Similarly as for the temperature constraint \eqref{eq:Ctemp}, we rewrite the constraint using a functional
\begin{align}\label{eq:consVM}
    \C_\mathrm{VM}(\Omega)=\int_{\tilde{D}}c_\Omega\Phi_{M^*}(s(u(\Omega)))\dx=0,\quad c_\Omega=\sum_{i\in\I}c_i\chi_{\Omega_i}
\end{align}
where $u(\Omega)$ is the solution of \eqref{eq:elasticity} for a given design $\Omega$, with weights $c_f=1,c_a=c_{m_1}=c_{m_2}=0$ and $\Phi_{M^*}$ is given by
\begin{align}\label{eq:Phi}
    \Phi_S(s)=\left(1+\left(\frac{s}{S}\right)^p\right)^\frac{1}{p}-1, \; S>0.
\end{align}
\begin{remark}
    Similar to the function $\Theta_S$ in \eqref{eq:Ctemp}, the function $\Phi_S$ is a regularization of the maximum function. For $\Theta_S$ we have the one-to-one correspondence $\Theta_S(s)=0\Leftrightarrow s\le S$. For $\Phi_S$, this is no longer true, since $\Phi_S(s)>0$ for all $s>0$. Nevertheless, as carried out in \cite{Amstutz_Constraint}, the penalty function of a gradient state constraint (like stresses) has to fulfill a growth condition to ensure the existence of the topological derivative. Therefore, we can not employ $\Theta_S$ and continue with $\Phi_S$.
\end{remark}
\begin{remark}
    In \cite{Amstutz_vonMises}, the authors identified the weights with the Young's modulus $c_\Omega=E_\Omega$ which leads to a penalization of stresses also outside the target domain since $E_a>0$. As remarked in \cite{Holley_Diss}, this has a regularizing effect on the design optimization problem, which is needed when minimizing the volume of a mechanical device subject to local stress constraints. However, since we are doing multi-physical optimization with additional quantities of interest, we can continue with our choice of $c_\Omega$.
\end{remark}
\subsection{Topological derivative of the Von Mises stress constraint}
 The topological derivative of $\C_\mathrm{VM}$ was introduced in \cite{Amstutz_vonMises}. We only have to insert the weight $c_\Omega$ accordingly. First, we introduce the adjoint state $p$ which is the solution of
 \begin{align*}
    -\mathrm{div}\,\sigma_\Omega(p)&=\chi_{\tilde{D}}\mathrm{div}\left(c_\Omega\Phi_{M^*}'(s(u))\tilde{\mathbb{B}}_\Omega\sigma(u)\right)&&\text{ in }D\cup D_\mathrm{RI}\\
    \sigma_\Omega(p)n&=-c_\Omega\Phi_{M^*}'(s(u))\tilde{\mathbb{B}}_\Omega\sigma(u)&&\text{ on }\Gamma_R\\
    p&=0&&\text{ on }\Gamma_\mathrm{SH}\\
    R_{\frac{\pi}{4}}p|_{\Gamma_1}&=p|_{\Gamma_2},
\end{align*}
with
\begin{align*}
    \Phi'_{S}(s)=\frac{s^{p-1}}{S^p}\left(1+\left(\frac{s}{S}\right)^p\right)^\frac{1-p}{p},\quad\tilde{\mathbb{B}}_\Omega\sigma=6\mu_\Omega\sigma+(\lambda_\Omega-2\mu_\Omega)\mathrm{tr}\sigma.
\end{align*}
For $z\in\Omega_i$, we use the abbreviations $S^u_z=\sigma(u(z)), \E^p_z=\varepsilon(p(z)), S^\mathrm{VM}_z=\tfrac{1}{2}\mathbb{B}S^u_z:S^u_z, P_z=p(z)$. The topological derivative from material $i\in\I$ to $j\in\I, i\neq j$ reads
\begin{align}
    \begin{aligned}
        \id^{i\rightarrow j}\C_\mathrm{VM}(\Omega)(z)&=3E_ir^{i\rightarrow j}S^u_z:\E^p_z-(\rho_j-\rho_i)V_\mathrm{rot}z\cdot p_z\\&+\chi_{\tilde{D}}\bigg(c_i\Psi_{M^*}^{i\rightarrow j}(S^u_z)+(r^{i\rightarrow j})^2c_i\Phi_{M^*}'\left(S^\mathrm{VM}_z\right)(S^\mathrm{VM}_z+S^u_z:S^u_z)\\&\phantom{aaaaaa}+c_j\Phi_{M^*}((1-2r^{i\rightarrow j})^2S^\mathrm{VM}_z)-c_i\Phi_{M^*}(S^\mathrm{VM}_z)+4r^{i\rightarrow j}c_i\Phi'_{M^*}(S^\mathrm{VM}_z)\bigg),
    \end{aligned}
\end{align}
where $r^{i\rightarrow j}=\frac{E_j-E_i}{2E_j+E_i}$ and $\Psi_{M^*}^{i\rightarrow j}(\sigma)=\Psi_{r^{i\rightarrow j}}(\sigma/\sqrt{M^*})$ with
\begin{align}
\begin{aligned}
    \Psi_r(\sigma)&=\frac{1}{\pi}\int_0^1\int_0^\pi\left(\Phi_1(\tfrac{1}{2}\mathbb{B}\sigma:\sigma+\Lambda_r(t,\varphi))-\Phi_1(\tfrac{1}{2}\mathbb{B}\sigma:\sigma)-\Phi_1'(\tfrac{1}{2}\mathbb{B}\sigma:\sigma)\Lambda_r(t,\varphi)\right)t^{-2}\,\mathrm{d}\varphi\,\mathrm{d}t,\\
    \Lambda_r(t,\varphi)&=\frac{rt}{2}(5(\sigma_1^2-\sigma_2^2)\cos(\varphi)+3(\sigma_1-\sigma_2)^2(2-3t)\cos(2\varphi))\\&+\frac{r^2t^2}{4}(3(\sigma_1+\sigma_2)^2+6(\sigma_1^2-\sigma_2^2)(2-3t)\cos(\varphi)+(\sigma_1-\sigma_2)^2(3(3t-2)^2+4\cos(\varphi)^2)),
\end{aligned}
\end{align}
where $\sigma_1,\sigma_2$ are the principal stresses of $\sigma$, i.e. the eigenvalues of $\sigma$. Since the Von Mises stresses are frame invariant, i.e. they are uniquely defined by the eigenvalues $\sigma_1,\sigma_2$, the function $\Psi_r(\sigma)$ depends only on three parameters $r,\sigma_1,\sigma_2$ and can be precomputed for selected sample values.
\begin{remark}
    We restricted ourselves to a fixed Poisson ratio $\nu=\tfrac{1}{3}.$ In \cite{Amstutz_vonMises}, the authors considered an arbitrary but fixed value for all materials. The derivation of the topological derivative of the Von Mises stress constraint considering a material dependent Poisson ratio $\nu_\Omega$ is also covered by the framework from \cite{Gangl_Automated}. However, the computation of the explicit formula is to our knowledge still open.
\end{remark}
\section{Numerical results} \label{sec_numerics}
\begin{table}
    \centering
    \begin{tabular}{c|c|c|c|c|c|c|c}
         $k_\mathrm{max}$&$\varepsilon$&$s_\mathrm{min}$&$s_\mathrm{max}$&$\gamma$&$\delta$&$\Delta_\ell$&$\varepsilon_\ell$\\
         \hline
         1000&$4^\circ$&0.05&1&0.5&1.5&$10^5$&$10^{-3}V^*$\\
    \end{tabular}
    \caption{Parameters used in the level set algorithm (Algorithm \ref{alg:Volume})}
    \label{tab:params}
\end{table}
We now present the application of the level set based topology optimization algorithm, Algorithm \ref{alg:Volume}, with parameters as presented in Table \ref{tab:params} to the different optimization problems introduced above. We perform the optimizations using the open source Finite Element package NGSolve \cite{ngsolve}, an implementation in C++ with a very flexible Python interface. For spatial discretization, we use lowest order Lagrangian basis functions on a triangular mesh with 3567 nodes. As proposed in \cite{Amstutz_Levelset}, we use the same continuous piecewise linear functions to represent both, the states and the level set function $\psi$. We do not adapt the mesh throughout the optimization and assign the cut ratio, i.e. the relative volume of the different materials, to the indicator functions $\chi_{\Omega_i}$ on elements which are cut by the level set function $\psi$. They are used to compute all piecewise defined quantities: electromagnetic material laws $h_\Omega,\sigma_\Omega$ \eqref{eq:material_laws}, mechanical material parameters $E_\Omega,\nu_\Omega,\rho_\Omega$ \eqref{eq:sigmaOmega}, eddy-current loss density $P_\Omega^\mathrm{ed}$ \eqref{eq:POmega}, weights $c_\Omega$ in \eqref{eq:Ctemp}, \eqref{eq:consVM} and topological derivative \eqref{eq:TDmulti} which corresponds to a linear interpolation weighted by the volume ratios. The sensitivities were smoothed by \eqref{eq:smoothing} with a parameter $\rho=10^{-6}$. We use $K=30$ harmonic basis functions for the Lagrange multiplier \eqref{eq:mortar}, \eqref{eq:quasimortar} and $N=11$ rotor positions \eqref{eq:angles} to solve the magnetic problems. This results in a problem with 39567 degrees of freedom for the all-in-one system \eqref{eq:K_allinone}, which is still suitable for direct methods. To solve the exterior problem \eqref{eq:exterior}, we truncate the unbounded domain by a centered ball of radius 128 and use lowest order Finite Elements on an adaptively refined triangular mesh with 61272 nodes. In an offline phase we evaluate \eqref{eq:TDevaluate} for $50$ radial (i.e. $t=|U|$) and 48 angular ($\beta=\angle U$) samples, which we interpolate by a cubic bivariate spline in order to evaluate the topological derivative in the online phase.

We start all optimization runs with the design shown in Figure \ref{fig:ini} and constrain the overall magnet volume to 10\% of the rotor, i.e. $V^*=0.1|D\cup D_\mathrm{RI}|$, which is the magnet volume of the initial design. We evaluate all designs by computing their average torque $\overline{\T_\mathrm{ed}}(\Omega)$ \eqref{eq:torque} and maximal temperature $\vartheta_\mathrm{max}(\Omega):=\max_{x\in \Omega_{m_1}\cup\Omega_{m_2}}\vartheta(x)$ based on the solution of \eqref{eq:magnetothermal} and the maximal Von Mises stress $s_\mathrm{max}(\Omega):=\max_{x\in\Omega_f}s(u(x))$ \eqref{eq:vonMises}. The results are collected in Table \ref{tab:results}. In the figures we show the final designs (left) and the corresponding eddy-current loss densities scaled to $50\,\mathrm{MW}/\mathrm{m}^2$ (mid) and Von Mises stress distributions scaled to $500\,\mathrm{MPa}$ (right).
\begin{figure}
    \centering
    \includegraphics[width=0.54\textwidth,trim=336 42 336 42,clip]{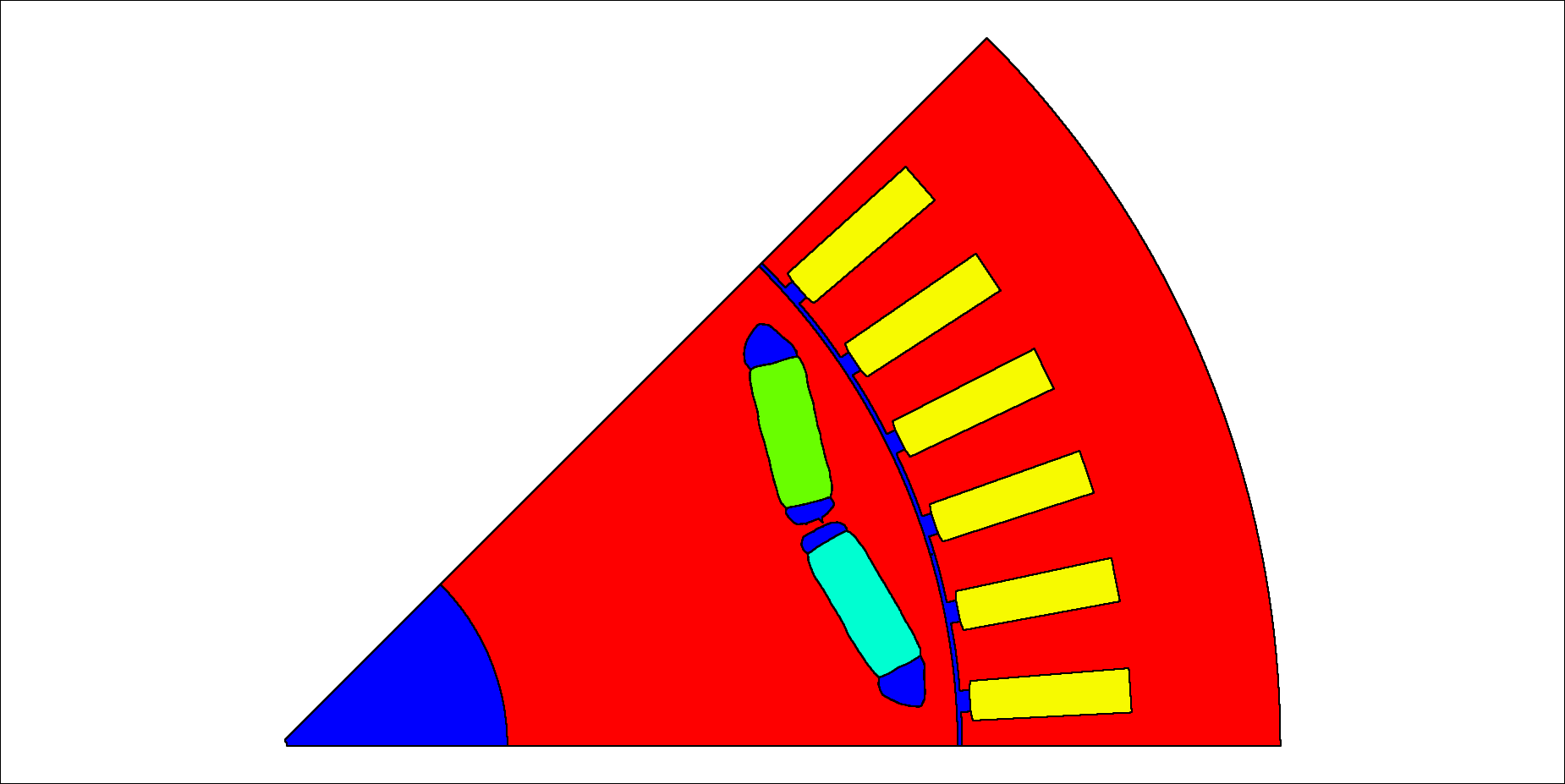}\quad
    \includegraphics[width=0.18\textwidth,trim=880 42 700 300,clip]{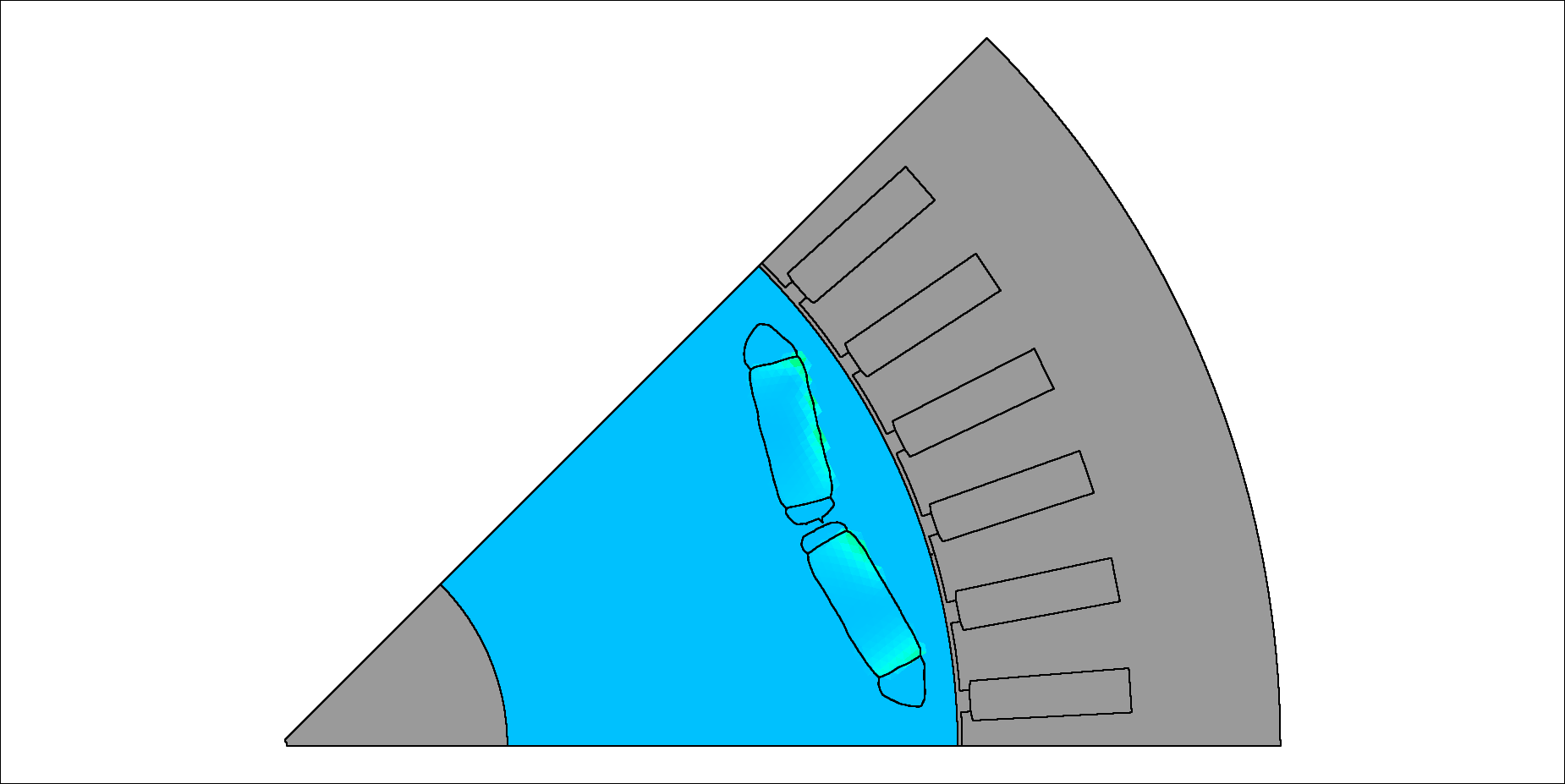}\quad
    \includegraphics[width=0.18\textwidth,trim=880 42 700 300,clip]{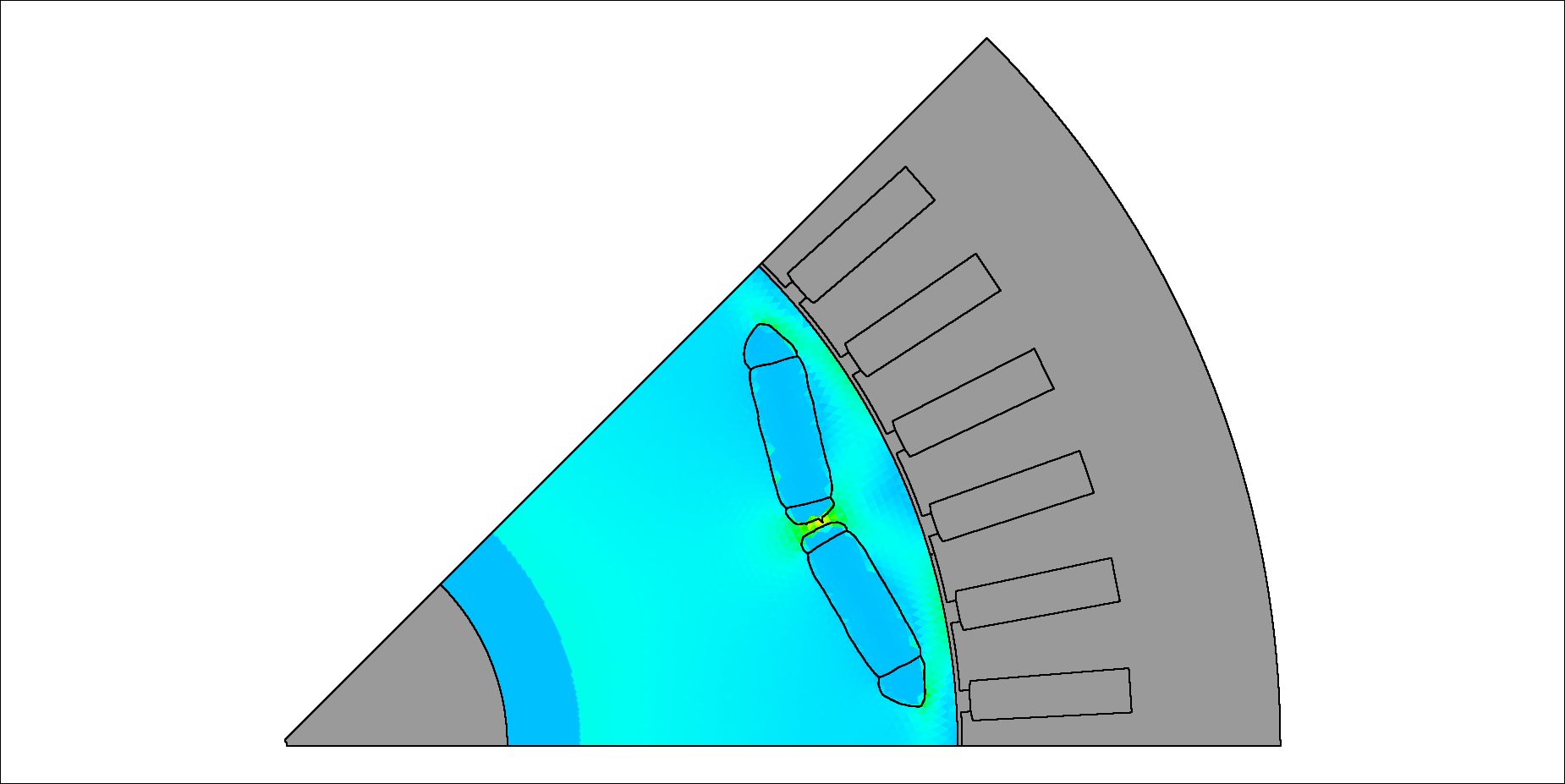}
    \caption{Initial design $\Omega_\mathrm{ini}$ (left), average loss density (mid), Von Mises stress distribution (right). Average torque $\overline{\T_\mathrm{ed}}(\Omega_\mathrm{ini})=653\mathrm{\,Nm}$, maximal temperature $\vartheta_\mathrm{max}(\Omega_\mathrm{ini})=72^\circ\mathrm{C}$, maximal Von Mises stress $s_\mathrm{max}(\Omega_\mathrm{ini})=340\,\mathrm{MPa}$.}
    \label{fig:ini}
\end{figure}
\subsection{Torque maximization}
First, we optimized the cost functional \eqref{eq:CostTorque}, i.e. maximized the average torque neglecting eddy currents. After 9 iterations of Algorithm \ref{alg:Volume} this led to the design $\Omega_\T$, shown in Figure \ref{fig:T}, with an average torque of $\overline{\T_\mathrm{ed}}(\Omega_\T)=858$\,Nm and a maximal temperature of $198^\circ$C. The design $\Omega_{\T_\mathrm{ed}}$ obtained by maximizing the torque considering the magnetoquasistatic equation \eqref{eq:CosttorqueED} is very close to the previous design $\Omega_\T$ and is thus not shown. Its torque is very similar $\overline{\T_\mathrm{ed}}(\Omega_{\T_\mathrm{ed}})=856$\,Nm, see Table \ref{tab:results} and we observe a slight decrease of the maximal temperature to $184^\circ$C due to the fact, that the eddy currents were considered in the simulations. Looking on the energy balance of the system, it is beneficial for the design to reduce the eddy-current losses to have most electrical power converted to mechanical power, i.e. the torque.
\begin{figure}
    \centering
    \includegraphics[width=0.54\textwidth,trim=336 42 336 42,clip]{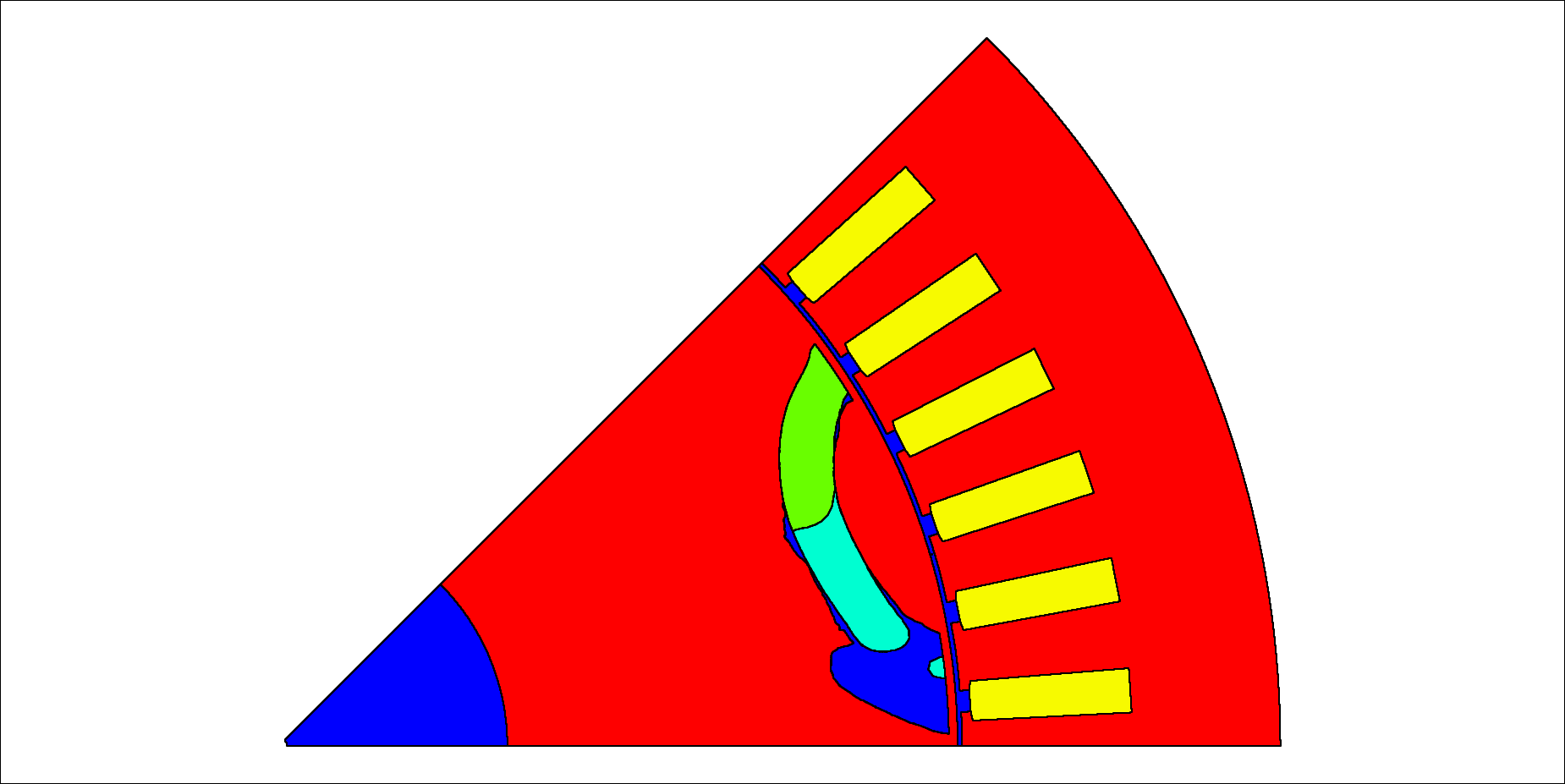}\quad
    \includegraphics[width=0.18\textwidth,trim=880 42 700 300,clip]{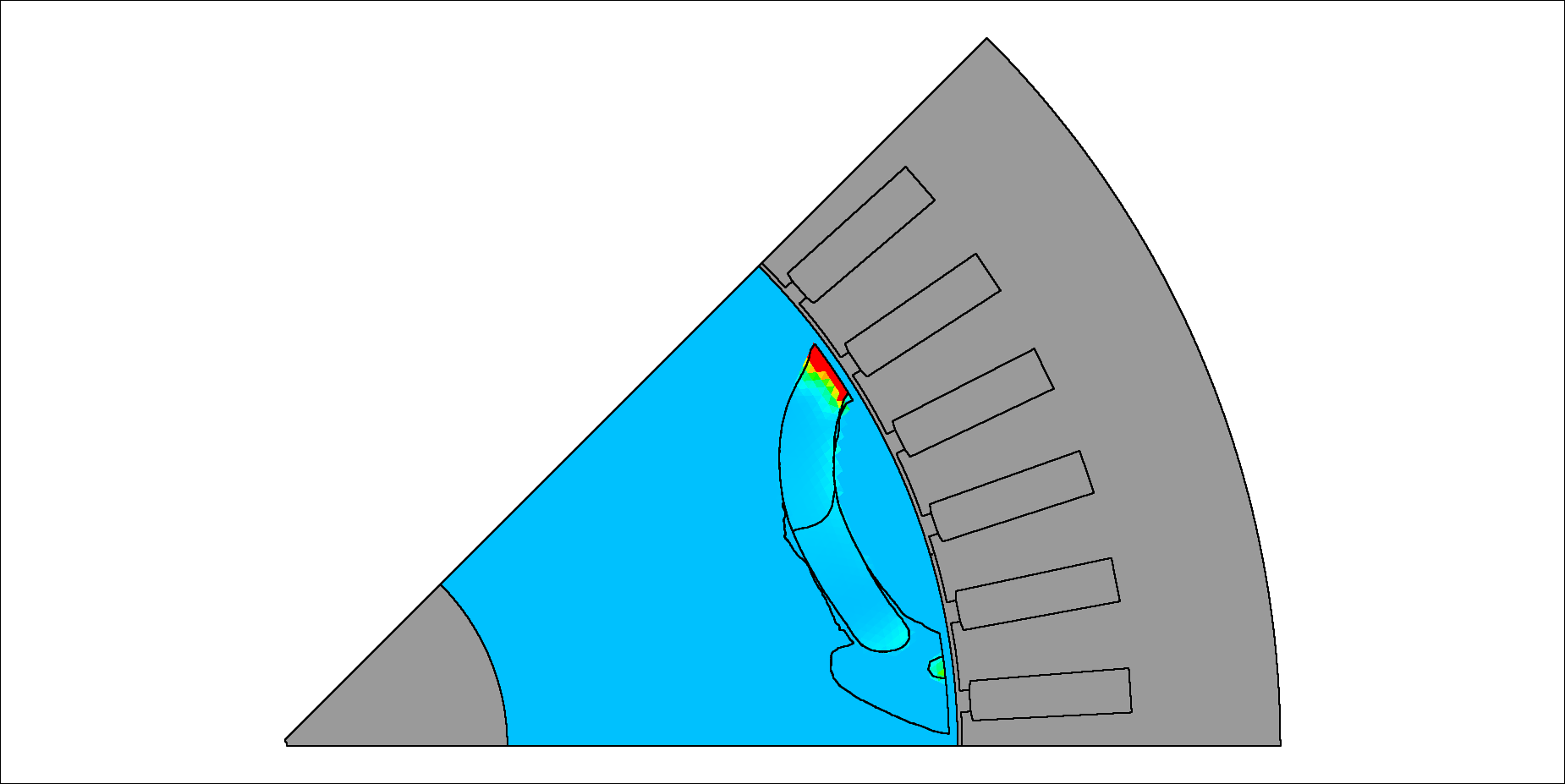}\quad
    \includegraphics[width=0.18\textwidth,trim=880 42 700 300,clip]{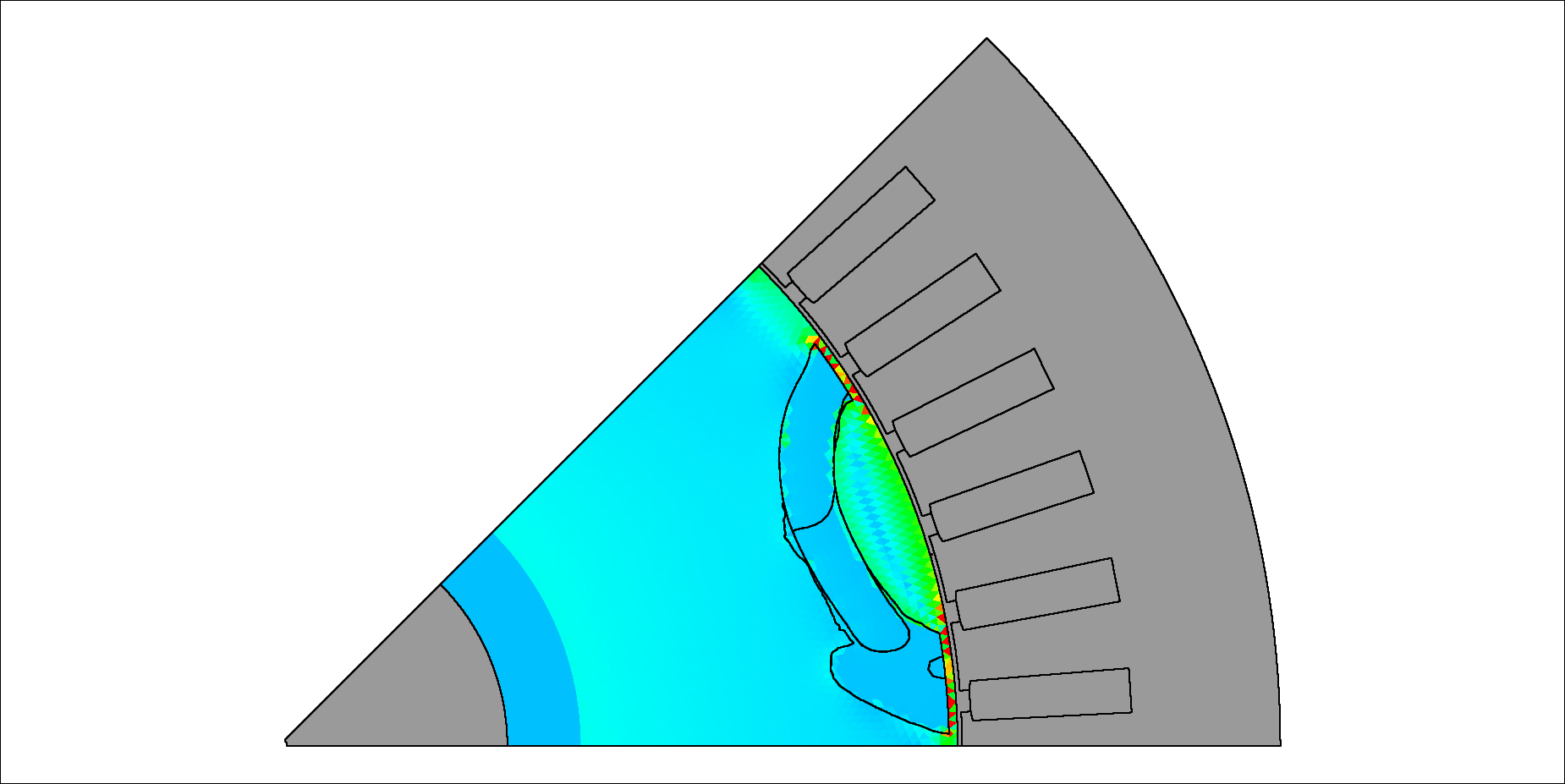}
\caption{Final design $\Omega_\T$ of maximizing average torque subject to magnetostatics \eqref{eq:CostTorque} (left), average loss density (mid), Von Mises stress distribution (right). Average torque $\overline{\T_\mathrm{ed}}(\Omega_\T)=858\,\mathrm{\,Nm}$, maximal temperature $\vartheta_\mathrm{max}(\Omega_\T)=198^\circ\mathrm{C}$, maximal Von Mises stress $s_\mathrm{max}(\Omega_\T)=772\,\mathrm{MPa}$.}    \label{fig:T}
\end{figure}
\subsection{Temperature constraint}
Next, we added the temperature constraint \eqref{eq:Ctemp} with a limiting temperature $T^*=90^\circ\mathrm{C}$. This is realized by optimizing the functional $\J(\Omega)=\overline{\T_\mathrm{ed}}(\Omega)+10^6\C_t(\Omega)$. Algorithm \ref{alg:Volume} converged after 58 iterations yielding the design presented in Figure \ref{fig:TcT}. The permanent magnets avoid regions close to the air gap, where the eddy-current loss density was very high in the unconstrained optimization, see Figure \ref{fig:T} (mid). This comes with a reduction of the average torque to $\overline{\T_\mathrm{ed}}(\Omega_{\T_\mathrm{ed},\C_t})=841\mathrm{\,Nm}$, which is $-2.0\%$ compared to the unconstrained design $\Omega_\T$. The corresponding maximal temperature $\vartheta_\mathrm{max}(\Omega_{\T_\mathrm{ed},\C_t})=92^\circ\mathrm{C}$ slightly violates the constraint. As mentioned in \cite{AndradeNovotnyLaurain2024}, this is due to using a finite weight, in our case $10^6<\infty$, to incorporate the constraint.
\begin{figure}
    \centering
    \includegraphics[width=0.54\textwidth,trim=336 42 336 42,clip]{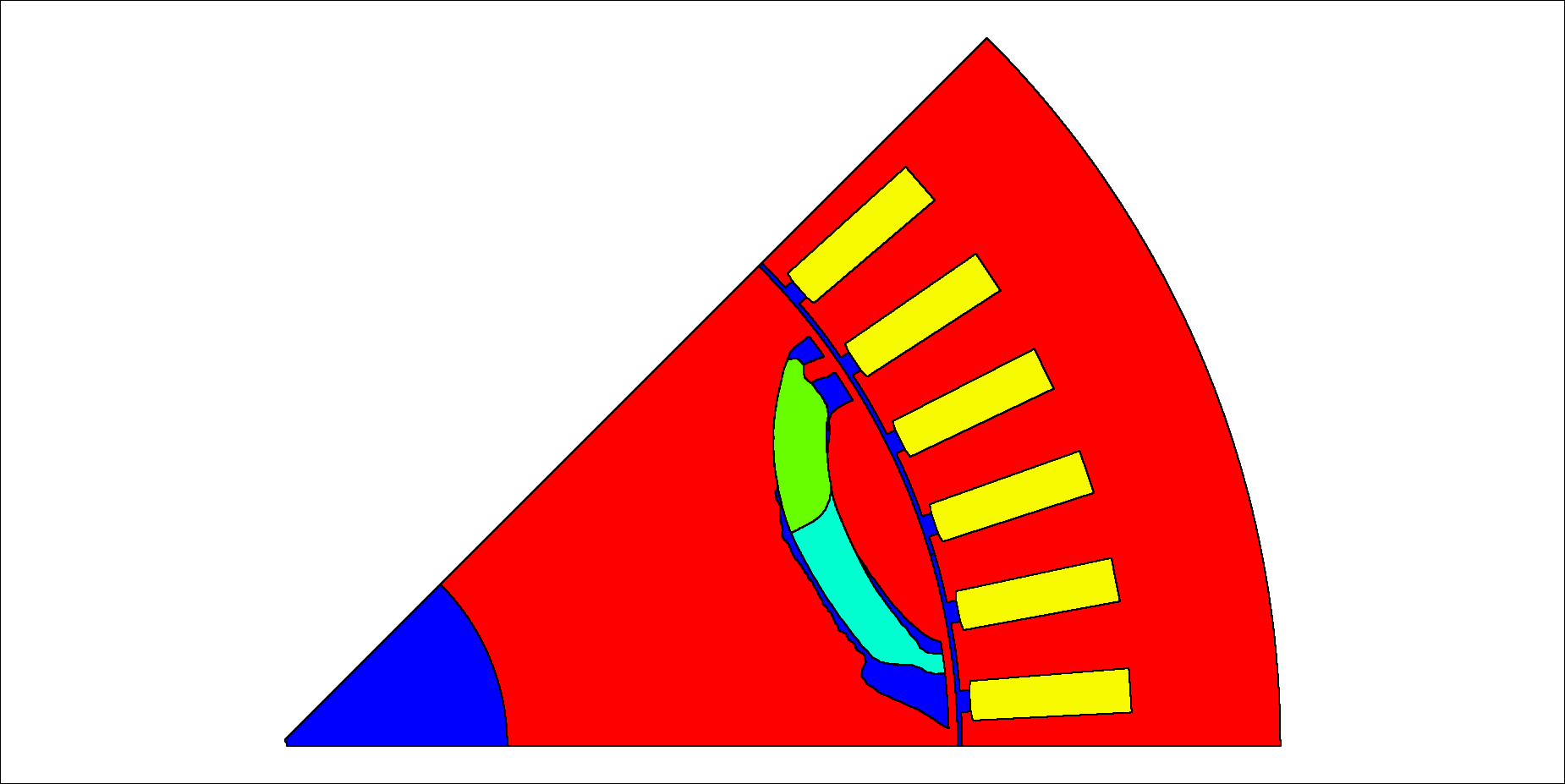}\quad
    \includegraphics[width=0.18\textwidth,trim=880 42 700 300,clip]{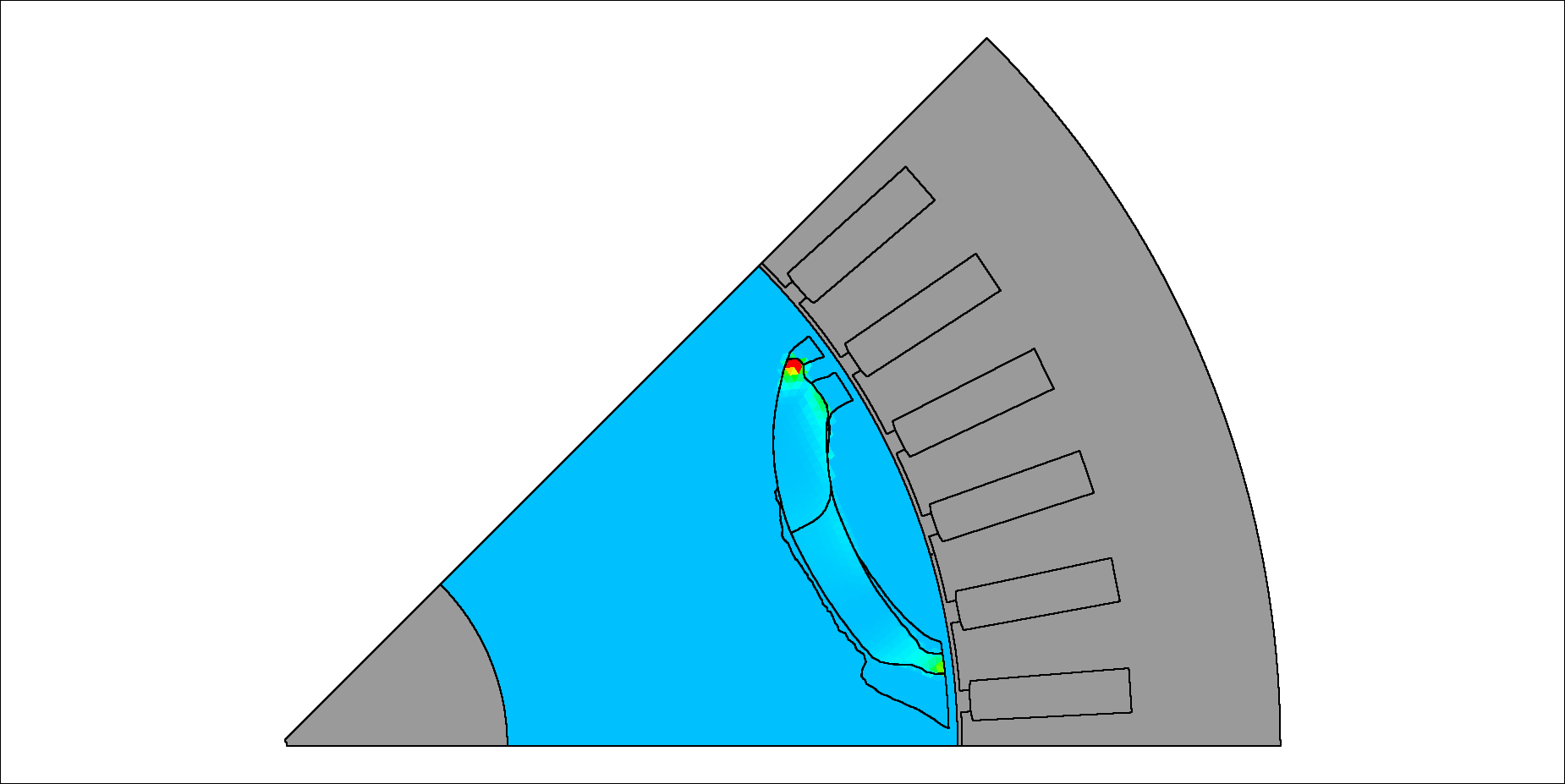}\quad
    \includegraphics[width=0.18\textwidth,trim=880 42 700 300,clip]{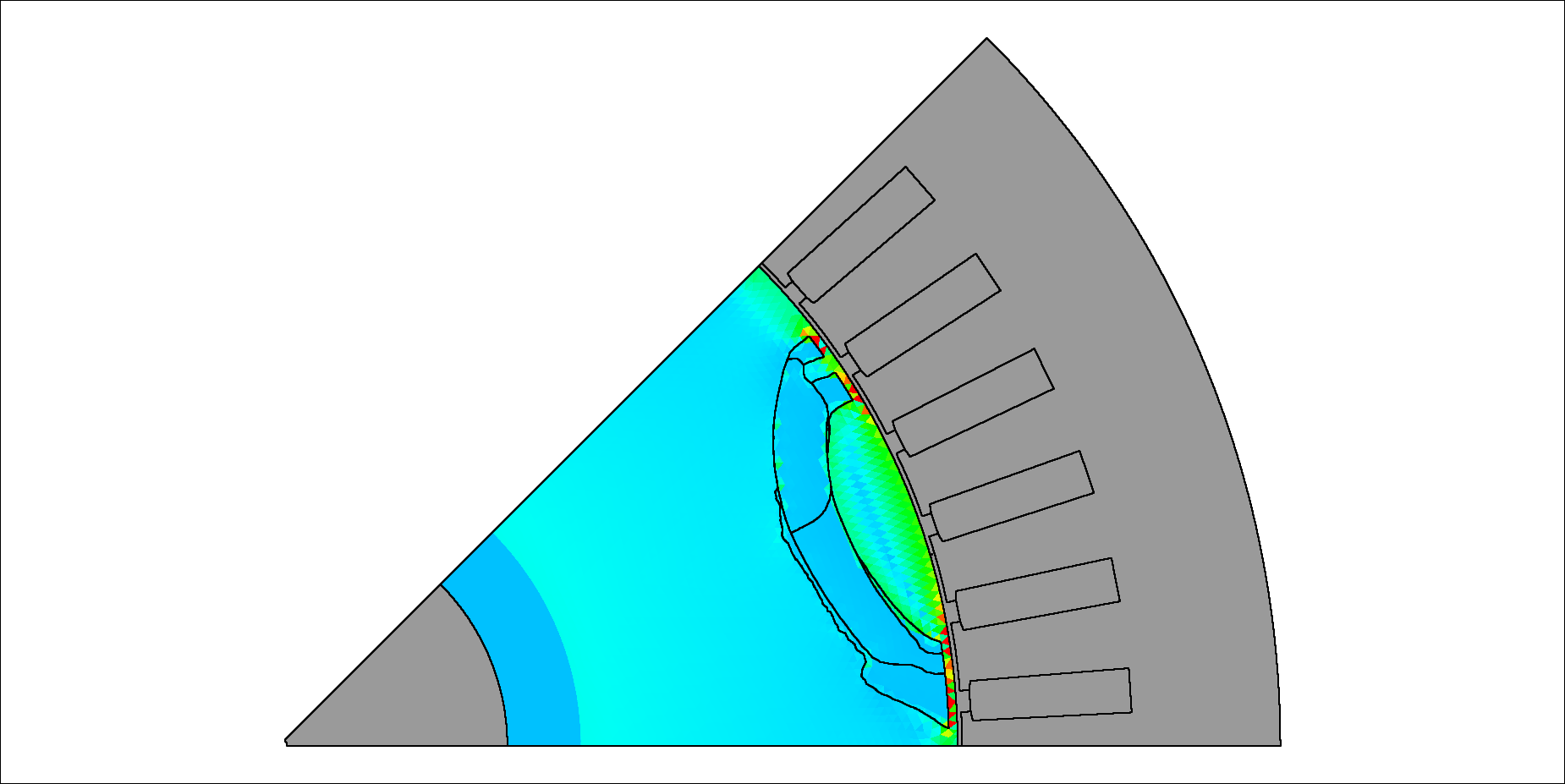}
\caption{Final design $\Omega_{\T_\mathrm{ed},\C_t}$ of maximizing average torque subject to magnetoquasistatics \eqref{eq:CosttorqueED} with temperature constraint \eqref{eq:Ctemp} (left), average loss density (mid), Von Mises stress distribution (right). Average torque $\overline{\T_\mathrm{ed}}(\Omega_{\T_\mathrm{ed},\C_t})=841\mathrm{\,Nm}$, maximal temperature $\vartheta_\mathrm{max}(\Omega_{\T_\mathrm{ed},\C_t})=92^\circ\mathrm{C}$, maximal Von Mises stress $s_\mathrm{max}(\Omega_{\T_\mathrm{ed},\C_t})=792\,\mathrm{MPa}$.}    \label{fig:TcT}
\end{figure}
\subsection{Temperature and stress constraints}
In the right plot of Figures \ref{fig:T} and \ref{fig:TcT} the Von Mises stress distribution is displayed. In the iron ring, the stresses exceed the yield stress $S^*=500\mathrm{MPa}$ which would lead to mechanical failure. We additionally consider the pointwise stress constraint \eqref{eq:consVMPW} by minimizing the functional $\J(\Omega)=\overline{\T_\mathrm{ed}}(\Omega)+10^6\C_t(\Omega)+3\cdot10^7\C_\mathrm{VM}(\Omega)$ with functionals as defined in \eqref{eq:CosttorqueED}, \eqref{eq:Ctemp}, \eqref{eq:consVM}, respectively. We obtained design $\Omega_{\T_\mathrm{ed},\C_t,\C_\mathrm{VM}}$, shown in Figure \ref{fig:all} (right) after 179 iterations of Algorithm \ref{alg:Volume}. Similarly as for $\Omega_{\T_\mathrm{ed},\C_t}$ we have a negligible violation of the temperature limit $T^*=90^\circ\mathrm{C}$ by $\vartheta_\mathrm{max}(\Omega_{\T_\mathrm{ed},\C_t,\C_\mathrm{VM}})=91^\circ\mathrm{C}$. The maximal Von Mises stress is reduced from almost $800\,\mathrm{MPa}$, for the previously optimized designs, to $s_\mathrm{max}(\Omega_{\T_\mathrm{ed},\C_t,\C_\mathrm{VM}})=464\,\mathrm{MPa}<S^*$. The overrating of the constrained can be explained by using a finite $p=16<\infty$ to regularize the maximum function in \eqref{eq:Phi}. Table \ref{tab:results} shows a comparison between the quantities of interest (average torque, maximum temperature, maximum Von Mises stresses) for all optimized designs.
\begin{figure}
    \centering
    \includegraphics[width=0.54\textwidth,trim=336 42 336 42,clip]{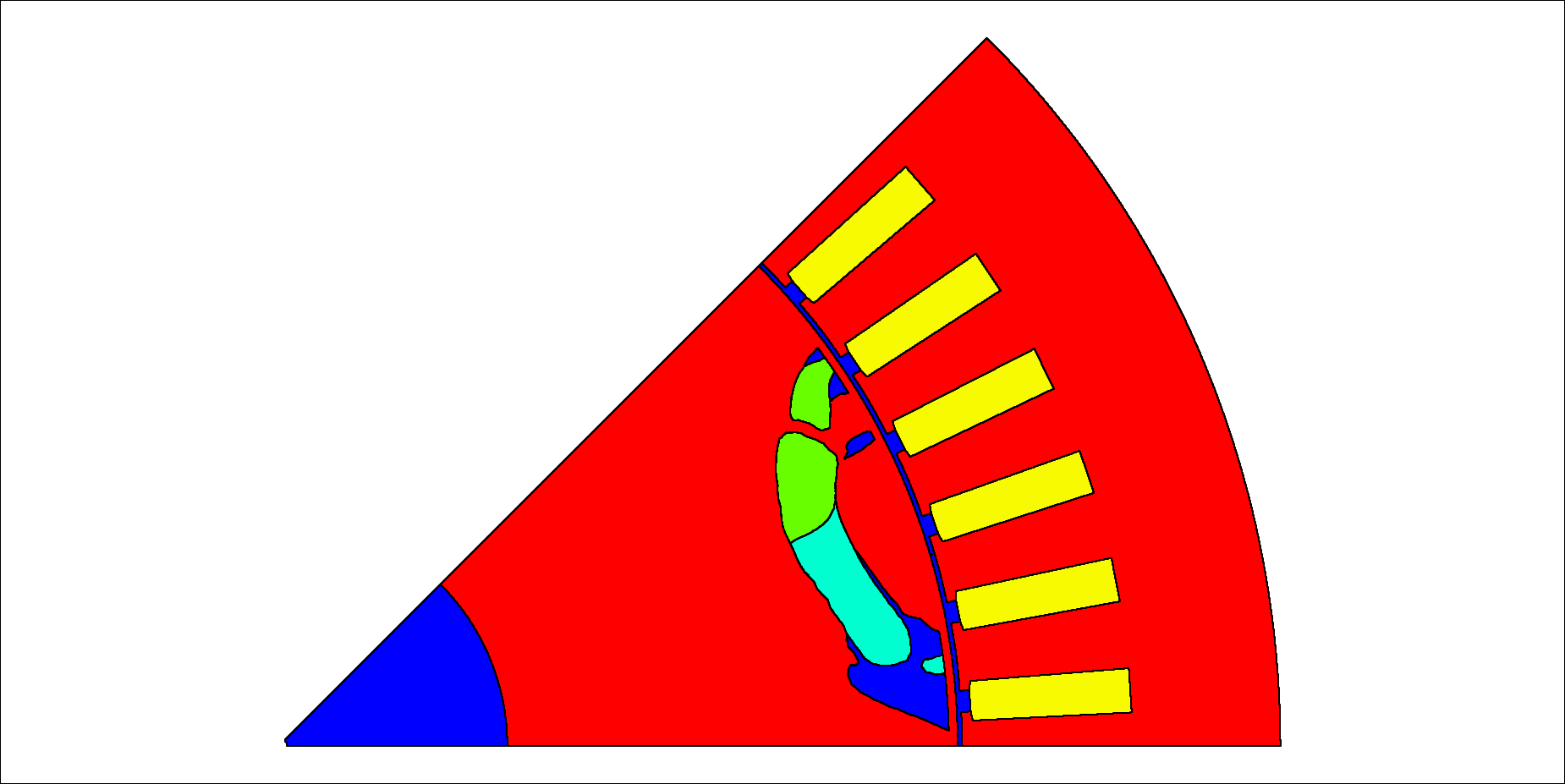}\quad
    \includegraphics[width=0.18\textwidth,trim=880 42 700 300,clip]{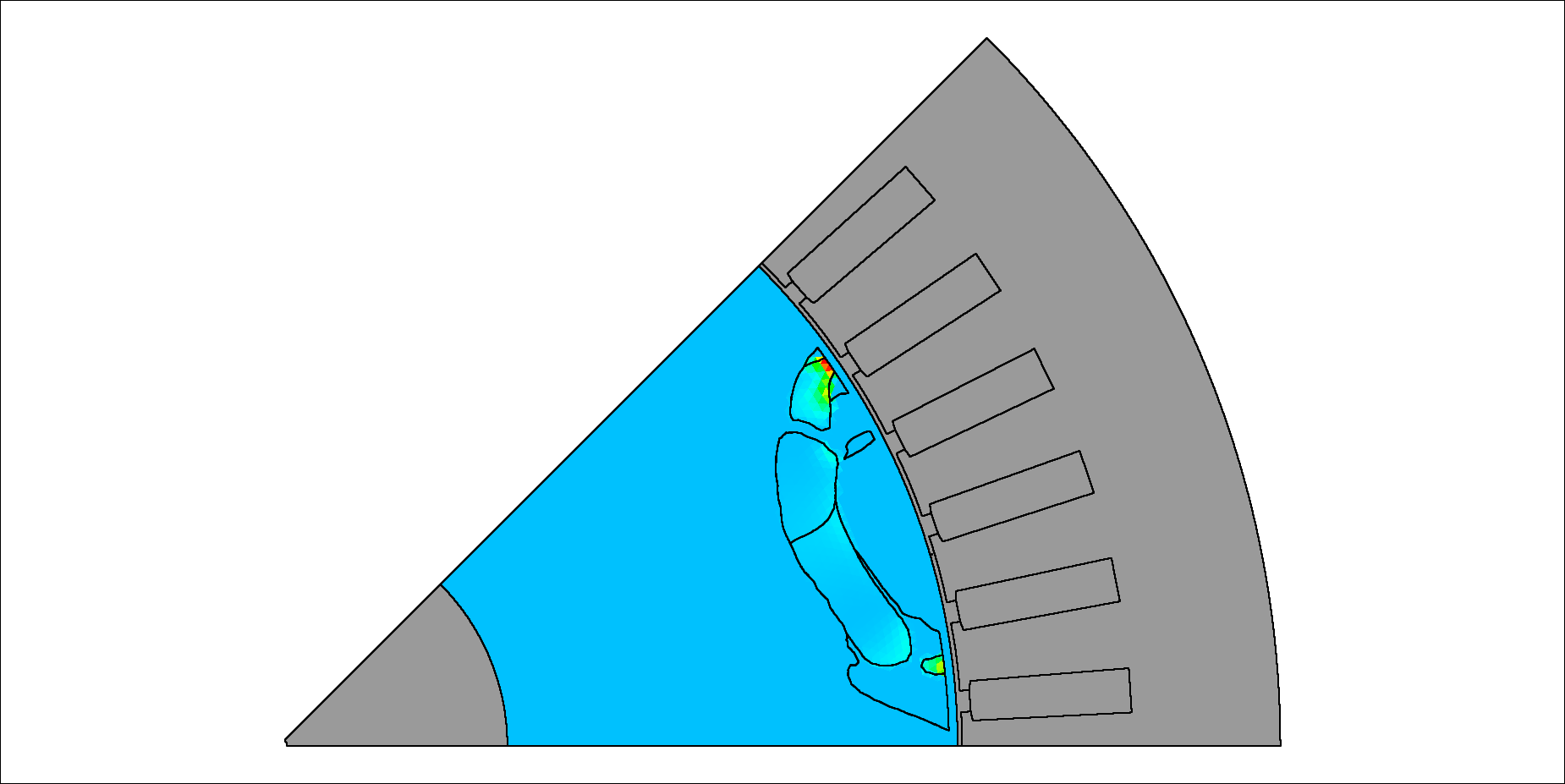}\quad
    \includegraphics[width=0.18\textwidth,trim=880 42 700 300,clip]{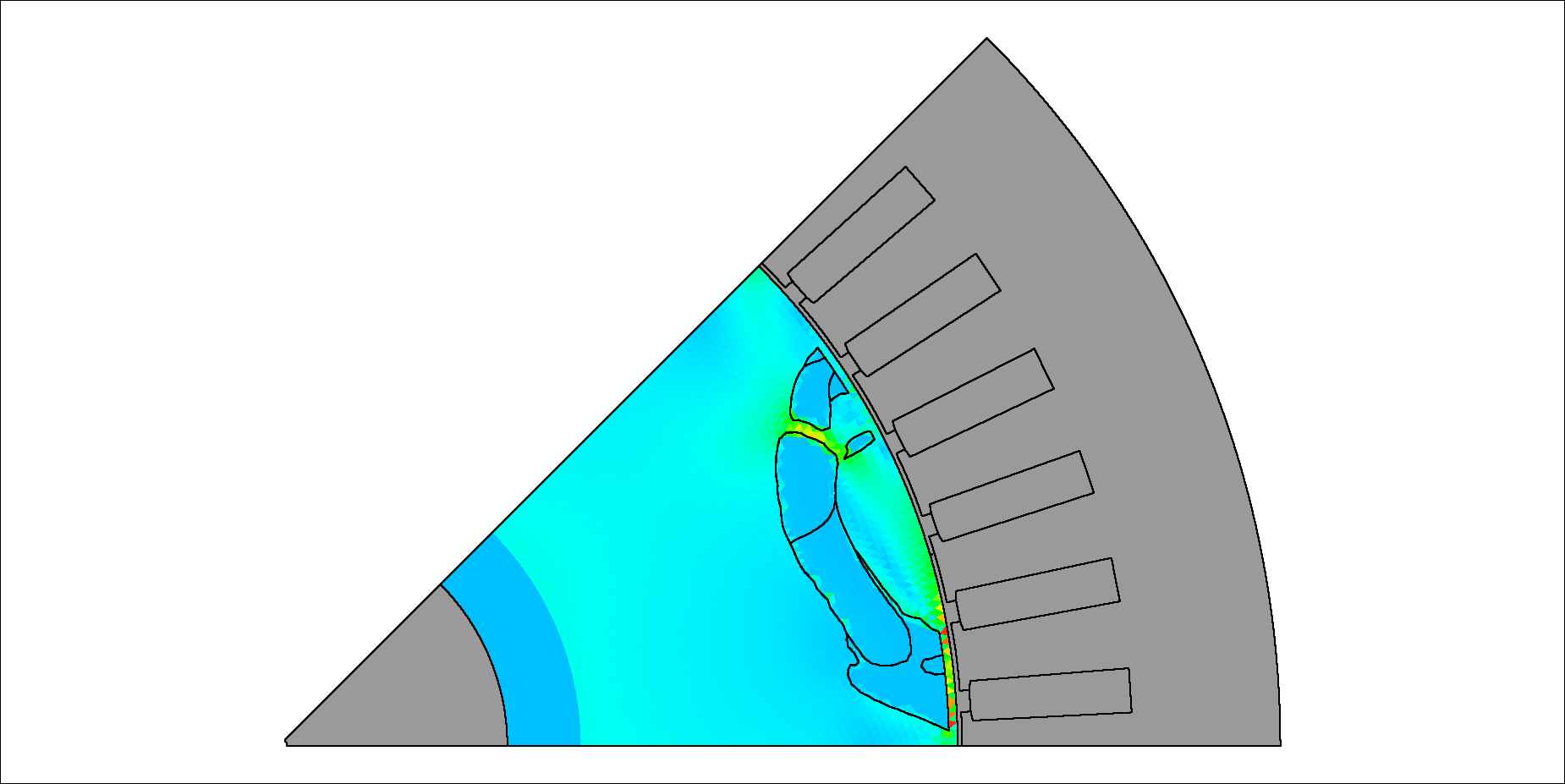}
\caption{Final design $\Omega_{\T_\mathrm{ed},\C_t,\C_\mathrm{VM}}$ of maximizing average torque subject to magnetoquasistatics \eqref{eq:CosttorqueED} with temperature constraint \eqref{eq:Ctemp} and stress constraint \eqref{eq:consVM} (left), average loss density (mid), Von Mises stress distribution (right). Average torque $\overline{\T_\mathrm{ed}}(\Omega_{\T_\mathrm{ed},\C_t,\C_\mathrm{VM}})=825\mathrm{\,Nm}$, maximal temperature $\vartheta_\mathrm{max}(\Omega_{\T_\mathrm{ed},\C_t,\C_\mathrm{VM}})=91^\circ\mathrm{C}$, maximal Von Mises stress $s_\mathrm{max}(\Omega_{\T_\mathrm{ed},\C_t,\C_\mathrm{VM}})=464\,\mathrm{MPa}$.}    \label{fig:all}
\end{figure}

\include{name}
\begin{table}
    \centering
    \begin{tabular}{l|c|c|c|c}
         $\Omega$&iter&$\overline{T}_\mathrm{ed}(\Omega)$&$\vartheta_\mathrm{max}(\Omega)$&$s_\mathrm{max}(\Omega)$  \\\hline
         $\Omega_\mathrm{ini}$&-&653&72&340\\
         $\Omega_\T$&9&858&198&772\\
         $\Omega_{\T_\mathrm{ed}}$&9&856&184&785\\
         $\Omega_{\T_\mathrm{ed},\C_t}$&58&841&92&792\\
         $\Omega_{\T_\mathrm{ed},\C_t,\C_m}$&179&825&91&464\\
    \end{tabular}
    \caption{Evaluation of obtained designs. Iterations needed in Algorithm \ref{alg:Volume} to converge. Average torque considering magnetoquasistatics $\overline{\T_\mathrm{ed}}$, see \eqref{eq:CosttorqueED}. Maximal  temperature $\vartheta_\mathrm{max}$ based on \eqref{eq:heat} and maximal Von Mises stresses $s_\mathrm{max}$ according to \eqref{eq:vonMises}.}
    \label{tab:results}
\end{table}

\section{Conclusion and Outlook}
In this work, we applied the framework of topological derivatives to an electromagnetic-thermal coupled problem to optimize the performance an electric machine accounting for local temperature constraints. We developed a model to simulate the thermal behavior due to the eddy-current losses in the permanent magnets. Using a smooth penalty function, we were able to incorporate the constraint in our optimization. We want to mention that this required the tuning of only one additional parameter. The optimization was carried out using a multi-material level set algorithm, which is very flexible to incorporate even more materials. We presented a very efficient method to incorporate volume constraints in the context of design optimization problems yielding fast convergence. Although we considered constraints on volume, temperature and stresses, our algorithm was capable of converging after 179 iterations. The numerical results prove the effectiveness of our method and show the potential of free-form optimization in engineering applications. 

There are several possible extensions to think of. In our setting, we considered a one-way coupling, i.e., the temperature depended on the electromagnetic fields but not vice versa. One could model a temperature-dependent material behavior of the permanent magnet as done for parameter optimization in \cite{Babcock2024}. This could lead to a natural limitation of the maximal temperature by simply maximizing the average torque with respect to this two-way coupled problem allowing for even more design freedom. Furthermore, one could also include the averaging of the eddy currents directly in the equation, as proposed in \cite{Hameyer_Eddy}. The framework of considering local state constraints can be extended to other quantities, possibly including other physics, e.g., to constrain demagnetization \cite{Krenn_Demag}. To preserve manufacturability, we did not change the iron ring at the rotor surface and incorporated local stress constraints. Moreover, one could combine these stress constraints with connectivity constraints, allowing to extend the design domain to the air gap.

\paragraph*{Acknowledgments}
The work of the authors is partially supported by the joint DFG/FWF Collaborative Research Centre CREATOR (DFG: Project-ID 492661287/TRR 361; FWF: 10.55776/F90) at TU Darmstadt, TU Graz, JKU Linz and RICAM Linz. P.G. is partially supported by the State of Upper Austria.

\bibliographystyle{abbrv} 
\bibliography{TBmerged}

\end{document}